\documentclass[10pt]{amsart}

\usepackage{amsfonts,amscd,amsmath,amsthm,amssymb,amsfonts,latexsym,}
\usepackage[usenames,dvipsnames]{xcolor}
\usepackage{enumerate}
\usepackage{manfnt}
\usepackage[backref]{hyperref}

\title{Sufficient conditions for local tabularity of a polymodal logic}

\author{Ilya B.~Shapirovsky}

\address{New Mexico State University, USA}
  \email{ilshapir@nmsu.edu}

\newcommand\hide[1]\empty

\newtheorem{theorem}{Theorem}[section]
\newtheorem{proposition}[theorem]{Proposition}
\newtheorem{lemma}[theorem]{Lemma}
\newtheorem{corollary}[theorem]{Corollary}
\newtheorem{problem}[theorem]{Problem}

\newtheorem{claim}{Claim}[theorem]

\theoremstyle{definition}
\newtheorem{definition}[theorem]{Definition}
\newtheorem*{definition*}{Definition}

\def\Al{\mathrm{A}}
\def\AlA{\Al}
\def\AlB{\mathrm{B}}

\newcommand\framets[1]{#1}
\def\frI{\framets{I}}
\def\frF{\framets{F}}
\def\frG{\framets{G}}
\def\toto{\twoheadrightarrow}
\def\mM{\framets{M}}
\def\frA{\framets{A}}
\newcommand\LSum[2]{{\textstyle \sum_{#1}{#2}}}
\newcommand\LSuml[2]{{\textstyle {\sum\limits^{\lex}}_{#1}{#2}}}
\def\tiff{\text{ iff }}
\def\clI{\mathcal{I}}
\def\clF{\mathcal{F}}
\def\Log{\myoper{Log}}
\newcommand\myoper[1]{\mathop{\myopts{#1}}}
\newcommand\myopts[1]{\mathrm{#1}}
\def\lex{\mathrm{lex} }

\def\Di{\lozenge}
\def\Dih{{\Di^{\mathrm{h}}}}
\def\Div{{\Di^{\mathrm{v}}}}
\def\Boxh{\Box^{\mathrm{h}}}
\def\DiAl{\Di_\Al}
\def\imp{\rightarrow}

\def\GL{\LogicNamets{GL}}
\newcommand\LogicNamets[1]{\logicts{#1}}
\newcommand\logicts[1]{{\textsc{#1}}}
\newcommand\LS[1]{\LogicNamets{S#1}}
\newcommand\LK[1]{\LogicNamets{K#1}}
\def\vL{L}
\def\Alg{\myoper{Alg}}

\def\clU{\mathcal{U}}
\def\clV{\mathcal{V}}

\def\clC{\mathcal{C}}
\def\clD{\mathcal{D}}
\def\clK{\mathcal{K}}
\def\clV{\mathcal{V}}
\def\clS{\mathcal{S}}
\def\clT{\mathcal{T}}
\def\clP{\mathcal{P}}
\def\EE{\exists}
\def\AA{\forall}

\def\dom{\myoper{dom}}
\def\restr{{\upharpoonright}}
\newcommand\refl[1]{#1^{\mathrm{r}}}
\newcommand\irrefl[1]{#1^{\mathrm{ir}}}
\def\sub{\subseteq}
\newcommand\Def{\myoper{Def}}
\newcommand\ind{\myoper{ind}}
\newcommand\Cl{\myoper{Cl}}
\newcommand\Lind[2]{\myoper{Lind}(#1,#2)}
\newcommand\ov[1]{\overline{#1}}
\def\AllPart{\mathcal{A}}
\def\v{\theta}
\def\vext{\bar{\v}}

\def\AllDef{\mathcal{D}}

\def\up{\uparrow}
\def\rk{\myoper{rk}}

\def\AllCl{\mathcal{C}}
\def\vf{\varphi}
\def\mo{\vDash}
\def\vd{\vdash}
\def\con{\wedge}

\def\emp{\varnothing}
\def\SubFrs{\myoper{Sub}}

\def\se{\subseteq}
\newcommand\languagets[1]{\logicts{#1}}

\def\PV{\languagets{PV}}

\newcommand\trSub[2]{{\left[#1\right]}\restr{#2}}
\newcommand\reflTr[1]{{[#1]}_{\mathrm{r}}}

\newcommand\cone[2]{{#1}\!\left<#2\right>}

\begin{document}

\begin{abstract}
On relational structures and on polymodal logics, we describe operations which preserve local tabularity.
This provides new sufficient semantic and axiomatic conditions for local tabularity of a modal logic.
The main results are the following.

We show that local tabularity does not depend on reflexivity.
Namely, given a class $\clF$ of frames, consider the class $\refl{\clF}$
of frames, where the reflexive closure operation was applied to each relation in every frame in $\clF$.
We show that if the logic of $\refl{\clF}$ is locally tabular, then the logic of $\clF$ is locally tabular as well.

Then we consider the operation of sum on Kripke frames, where a family of frames-summands is  indexed by elements of another frame.
We show that if both the logic of indices and the logic of summands are locally tabular, then the logic of corresponding sums is also locally tabular.

Finally, using the previous theorem,
we describe an operation on logics that preserves local tabularity: we provide a set of formulas
such that the extension of the fusion of two canonical locally tabular logics with these formulas is locally tabular.

\end{abstract}

\maketitle

\section{Introduction}
We study local tabularity of normal propositional (poly)modal logics.

A logic $\vL$ is {\em locally tabular},
if, for every finite $k$, $L$ contains only a finite number of pairwise nonequivalent
formulas in a given $k$ variables.
Equivalently, a logic $\vL$ is locally tabular iff the variety of its algebras is locally finite, i.e., every $L$-algebra is locally finite.
Local tabularity is a strong property:
if a logic is locally tabular, then it has the finite model property (thus it is Kripke complete);
every extension of a locally tabular logic is locally tabular (thus it has the finite model property);
every finitely axiomatizable extension of a locally tabular logic is decidable.

Local tabularity is well-studied for the case of unimodal transitive logics (a logic is transitive, if it contains the formula $\Di\Di p\imp \Di p$, which expresses the transitivity of a binary relation).
According to the classical results by Segerberg \cite{Seg_Essay} and Maksimova \cite{Maks1975},
a transitive logic is locally tabular iff it contains a modal {\em formula $B_h$ of finite height} for some finite $h$.

In the non-transitive unimodal, and in the polymodal case, no axiomatic criterion of local tabularity is known.
We provide sufficient semantic and axiomatic conditions for local tabularity.

The paper is organized as follows.

In Section \ref{sec:prel}, we provide basic definitions and an exposition of some previous results.

In Section  \ref{sec:subframes}, we discuss {\em $k$-finite logics}:
we say that $L$ is {\em $k$-finite} for $k<\omega$, if
$L$ contains only a finite number of pairwise nonequivalent formulas in a given $k$ variables.
It is known \cite{Maksimova89} that every unimodal transitive 1-finite logic is locally tabular. Recently \cite{Glivenko2021} it was observed that without transitivity, this equivalence does not hold. In Section \ref{sec:subframes}, we give corollaries of $k$-finiteness and in particular, give a necessary axiomatic condition for $2$-finiteness of a logic.

In Section \ref{sec:refl}, we prove that local tabularity does not depend on reflexivity.
Namely, given a class $\clF$ of frames, consider the class $\refl{\clF}$
of frames, where the reflexive closure operation was applied to every relation in every frame in $\clF$.
We show that if the logic of $\refl{\clF}$ is locally tabular, then the logic of $\clF$ is locally tabular as well (the converse implication is trivial).

In Section 5,  we consider the operation of sum on Kripke frames, where a family of frames-summands is  indexed by elements of another frame.
 There is a number of transfer results for sums in modal logic: they relate to questions of complete axiomatization
\cite{Bekl-Jap,Balb2009,BalbDuque2016},
finite model property and decidability \cite{BabRyb2010,AiML2018-sums}, computational complexity \cite{ShapJapar08,CondSatOnline}.
Local tabularity also transfers under the operation of sum:
we show that if both the logic of indices and the logic of summands are locally tabular, then the logic of corresponding
sums is locally tabular. This fact implies the following axiomatic transfer result. Consider two disjoint alphabets $\AlA$ and $\AlB$ of modalities.
Let
$\Phi(\AlA,\AlB)$ be the set of all formulas
\begin{equation*}
\Dih \Div p\to \Div p, \; \Div\Dih p\to \Div p, \;\Div p\to \Boxh \Div p
\end{equation*}
with $\Div$ in $\AlA$ and $\Dih$ in $\AlB$. These formulas were considered earlier in \cite{Balb2009} in connection with axiomatization problems
for {\em lexicographic products of modal logics},
  and also in \cite{Bekl-Jap}
in the context of polymodal provability logic.
Let $L_1$ be a locally tabular logic in the language of $\Al$,
and let $L_2$ be a locally tabular logic in the language of $\AlB$.
Let $L_1*L_2$ be the fusion of $L_1$ and $L_2$,  that is, the smallest logic in the language of $\AlA\cup\AlB$ that contains $L_1\cup L_2$.
It is well-known that the fusion operation does not preserve local tabularity.
However, the extension of $L_1*L_2$ with the formulas $\Phi(\AlA,\AlB)$  remains locally tabular, provided that
this extension is Kripke complete, in particular, if $L_1$ and $L_2$ are canonical.

\section{Preliminaries}\label{sec:prel}
\subsection{Basic definitions} We assume that the reader is
familiar with basic notions in modal logic; see, e.g., \cite{CZ,BDV} for the references.
\subsubsection{Modal syntax} Fix a set $\Al$, the {\em alphabet of modalities}. We will always assume that $\Al$ is finite.

The set of {\em modal formulas over $\AlA$} is built from
a countable set of {\em variables} $\PV=\{p_0,p_1,\ldots\}$ using Boolean connectives $\bot,\imp$ and unary connectives $\Di\in \AlA$ ({\em modalities}).
Other logical connectives
are defined as abbreviations in the standard way, in particular $\Box\vf$
denotes $\neg \Di \neg \vf$  (perhaps with an index).

We will also use the following abbreviations: $\Di^0 \vf=\vf$,
$\Di^{i+1}\vf=\Di^i\Di\vf$,
$\Di^{\leq m} \vf =
\bigvee_{i\leq m} \Di^i \vf$, $\Box^{\leq m} \vf =\neg \Di^{\leq m} \neg \vf$.
Additionally, we write $\DiAl\vf$ for $\bigvee_{\Di\in\Al}\Di \vf$.

The term {\em unimodal} refers to the case when $\Al$ is a singleton.

A modal formula $\vf$ is said to be a {\em $k$-formula} for $k\leq\omega$, if every variable occurring in $\vf$ belongs to $\{p_i\mid i<k\}$.

\subsubsection{Relational and algebraic semantics}

An {\em $\AlA$-frame} is a structure  $\frF=(X,(R_\Di)_{\Di\in \AlA})$,
where $X$ is a set and ${R_\Di\se X{\times}X}$ for $\Di\in \AlA$.
We write $\dom{\frF}$ for $X$, the {\em domain of }$\frF$.\footnote{We do not require that $X\neq \emp$.}
For $a\in X$, $Y\subseteq X$, we put $R_\Di(a)=\{b\mid aR_\Di b\}$,
$R_\Di[Y]=\bigcup_{a\in Y} R_\Di(a)$.

Let $k\leq\omega$. A {\em $k$-valuation in a frame $F$} is a map $\{p_i\mid i<k\}\to \clP(X)$,
where $\clP(X)$ is the set of all subsets of $X$.
A {\em (Kripke) $k$-model on} $\frF$ is a pair $(\frF,\theta)$, where $\theta$ is a $k$-valuation.
The \emph{truth} of $k$-formulas in a $k$-model is defined in the usual way, in particular
${(\frF,\theta),a\mo\Di \vf}$ means that  ${(\frF,\theta),b\mo\vf}$  for some $b$ in $R_\Di(a)$.
We put
$$\vext(\vf)=\{a\mid (F,\v),a\mo\vf\}.$$
A $k$-formula $\vf$ is {\em true in a $k$-model $\mM$}, if $\mM,a\mo\vf$ for all $a$ in $\mM$.
A $k$-formula $\vf$ is {\em valid in a frame $\frF$}, in symbols $\frF\mo\vf$,
if $\vf$ is true in every $k$-model on $\frF$.

For a binary relation $R$ on a set $X$,
let $R^*$ denote its transitive reflexive closure $\bigcup_{i <\omega} R^i$, where
where $R^0$ is the diagonal $\{(a,a)\mid a\in X\}$ on $X$, $R^{i+1}=R\circ R^i$, and $\circ$ is the composition.
For an $\Al$-frame $\frF=(X,(R_\Di)_{\Di\in \Al})$, let $$R_\frF=\bigcup_{\Di\in \Al} R_\Di.$$
The {\em restriction  $\frF\restr Y$  of $\frF$ to its subset $Y$} is the frame $(Y, (R_\Di\cap (Y\times Y))_{\Di\in\Al})$.

The notions of p-morphism, generated subframe, disjoint sum of frames are defined in the standard way, see, e.g., \cite[Section 3.3]{CZ} or
\cite[Sections 2.1 and 3.3]{BDV}.

An {\em $\Al$-modal algebra} is a~Boolean algebra endowed
with unary operations $f_\Di$,  $\Di\in\Al$, that distribute w.r.t.~finite
disjunctions. A modal formula $\vf$ is {\em valid} in an algebra $B$, if $\vf=1$ holds in $B$.
The {\em (complex) algebra $\Alg{F}$ of an $\Al$-frame}  $\frF=(X,(R_\Di)_{\Di\in \AlA})$ is the
powerset Boolean algebra of $X$ endowed with the following unary operations $f_\Di$, $\Di\in\Al$:
$f_\Di(Y)=R_\Di^{-1}[Y]$ for $Y\subseteq X$. It is immediate that  the algebra of an $\Al$-frame is an $\Al$-modal algebra, and that
in a $k$-model $(\frF,\theta)$,  $\vext(\vf)$ is the value of $\vf$ under $\theta$ in the modal algebra $\Alg{F}$, see \cite[Section 5.2]{BDV} for details.

\smallskip

A formula is {\em valid in a class} $\clC$ of (relational or algebraic) structures, if it is  valid in every structure in $\clC$.
Validity of a set of formulas means validity of every formula in this set.

\subsubsection{Logics}
A ({\em propositional normal modal}) {\em logic} {\em in the language of $\Al$} is a set $\vL$ of formulas
over $\Al$
that contains all classical tautologies, the axioms $\neg\Di \bot$  and
$\Di  (p_0\vee p_1) \imp \Di  p_0\vee \Di  p_1$ for each $\Di$  in $\AlA$, and is closed under the rules of modus ponens,
substitution and {\em monotonicity}; the latter means that for each $\Di$ in $\AlA$, $\vf\imp\psi\in \vL$ implies $\Di \vf\imp\Di \psi\in \vL$.
We write $\vL\vd \vf$ for $\vf\in L$. For a set $\Phi$ of modal formulas over $\AlA$,
$\vL+\Phi$ is the smallest normal logic containing $\vL\cup\Phi$.

An {\em $L$-structure} (a frame or algebra) is a structure where $L$ is valid.

The set of all formulas that are valid in a class $\clC$ of (relational or algebraic) structures is called the {\em logic of} $\clC$ and is denoted by $\Log{\clC}$.
It is straightforward the  $\Log{\clC}$ is a propositional normal modal logic.
Moreover, any propositional normal modal logic $L$ is the logic of the class of $L$-algebras, see, e.g.,
\cite[Section 7.5]{CZ}. For a structure $B$, we write $\Log{B}$ for $\Log{\{B\}}$.

A logic $L$ is {\em Kripke complete}, if $L$ is the logic of a class of frames.
A logic has the {\em finite model property},
if it is the logic of a class of finite frames (by the cardinality of a frame, algebra, or model we mean the cardinality of its domain).

The {\em $k$-canonical model of} $\vL$ is built from
maximal $\vL$-consistent sets of $k$-formulas; the canonical relations and the valuation are defined in the standard way.  The following fact is well-known, see, e.g., \cite[Chapter 8]{CZ}.
\begin{proposition}\label{prop:k-canonical-model}[Canonical Model Theorem]
Let $\mM$ be the  $k$-canonical model of a logic $\vL$. Then
for all $k$-formulas $\vf$ we have:
$\vL\vd\vf$ iff $\vf$ is true in $M$.
\end{proposition}

\subsection{Locally tabular logics; Maltsev's criterion}
The algebraic notions required for this subsection can be found in, e.g., \cite[Chapter 2]{BurrisSankappanavar2012}.

Recall that an algebra $B$ is {\em locally finite}, if every finitely-generated
subalgebra of $B$ is finite.

Recall that a logic $\vL$ is {\em locally tabular},
if, for every $k<\omega$, $L$ contains only a finite number of pairwise nonequivalent
$k$-formulas. The latter means that
the countably generated free (or Lindenbaum) algebra $\Lind{L}{\omega}$ of $L$ is locally finite.
Equivalently, a logic $\vL$ is locally tabular iff the {\em variety of its algebras is locally finite}, i.e., every $L$-algebra is locally finite.

For $k<\omega$, we say that $L$ is {\em $k$-finite}, if
$L$ contains only a finite number of pairwise nonequivalent
$k$-formulas. Equivalently, $L$ is $k$-finite, if
the $k$-generated free algebra $\Lind{L}{k}$ of $L$ is finite, and equivalently,
the $k$-canonical model of $\vL$ is finite.

The following criterion of local finiteness was given by A. Maltsev \cite{Malcev73}.
A class  $\clC$ of algebras of a finite signature is said to be
{\em uniformly locally finite}, if
there exists a function $f:\omega\to \omega$ such that the cardinality of a subalgebra of any $B\in \clC$ generated by
$k<\omega$ elements does not exceed $f(k)$.
Local finiteness of the variety $\clK$  generated by $\clC$
is equivalent to uniform local finiteness of $\clC$ \cite[Section 14, Theorem 3]{Malcev73}. Moreover,
from the proof of this theorem it follows that if every $k$-generated subalgebra of any algebra
$\frA\in\clC$ has at most $n$ elements for a fixed finite $n$, then the $k$-generated free algebra of $\clK$ is finite. In terms of modal logic, we have:

\begin{theorem}[Corollary of \cite{Malcev73}]\label{thm:Maltsev}
Let $L$ be the logic of a class $\clF$ of $\Al$-frames.
\begin{enumerate}[\normalfont (1)]
\item \label{thm:Maltsev-kLF} $L$ is locally tabular iff $\{\Alg{F}\mid F\in\clF\}$ is uniformly locally finite.
\item \label{thm:Maltsev-kfin}
Let $k<\omega$. $L$ is $k$-finite iff there exists $n<\omega$ such that
for every $\frF\in \clF$, the cardinality of every subalgebra of $\Alg{F}$ generated by $k$ elements
is not greater than $n$.
\end{enumerate}
\end{theorem}

\subsection{Local tabularity, pretransitivity, and finite height}
In this subsection we discuss the criterion of local tabularity for unimodal logics of transitive frames given
in \cite{Seg_Essay,Maks1975}, and its corollaries for all modal logics \cite{CZ}, \cite{LocalTab16AiML}.

A poset is of {\em height} $h<\omega$,
if it contains a chain of $h$ elements and no chains of more than $h$ elements.

Recall that in an $\Al$-frame $\frF=(X,(R_\Di)_{\Di\in \Al})$, $R_\frF=\bigcup_{\Di\in \Al} R_\Di$.
A {\em cluster} in an $\Al$-frame $\frF$ is an equivalence class with respect to the relation $\sim_\frF\; =\{(a,b)\mid a R^*_\frF b \text{ and } b R^*_\frF a\}$. 
For clusters $C, D$, put $C\leq_\frF D$ iff $a R^*_\frF b$ for some $a\in C, b\in D$.
The poset $(X{/}{\sim_\frF},\leq_\frF)$ is called the {\em skeleton of} $\frF$.
The {\em height of a frame} $\frF$ 
is the height of its skeleton.

Consider the following formulas:
\begin{equation}
B_0=\bot,  \quad B_{h} = p_{h} \to  \Box(\Di p_{h} \lor B_{h-1}).
\end{equation}

In the unimodal transitive case, the following holds \cite{Seg_Essay}:
\begin{equation}\label{eq:fin-ht-trans}
(X,R)\mo B_h \text{ iff } \text{the height of $(X,R)$ is not greater than $h$}.
\end{equation}

A unimodal logic is called {\em transitive}, if it contains the formula $\Di\Di p\imp \Di p$, which validity in a frame $(X,R)$ means transitivity of $R$.
The following criterion of local tabularity for unimodal transitive logics
was given in \cite{Seg_Essay,Maks1975} (see also \cite[Section 12.4]{CZ}):
\begin{equation}\label{eq:thm-seg-maks}
\text{ A transitive logic is locally tabular iff it contains $B_h$ for some $h< \omega$.}
\end{equation}
In \cite{LocalTab16AiML}, this criterion was extended to unimodal logics
of relations such that $R^{m+1}\subseteq R\cup R^0$ for a fixed positive $m$: as well as in the transitive case, local tabularity is equivalent to finite height for such logics.
However, an axiomatic criterion of local tabularity for polymodal logics, and even for all unimodal logics is unknown.

Finite height and some weaker version of transitivity is a necessary, but in general not a sufficient condition for
local tabularity.  Namely, for a binary relation $R$, put
$R^{\leq m}=  \bigcup_{i \leq m} R^i$.
$R$ is called {\em $m$-transitive},  if $R^{\leq m}=R^*$.
$R$ is called {\em pretransitive}, if it is $m$-transitive for some finite $m$.
It is straightforward that
\begin{equation}
\text{$R$ is $m$-transitive iff
$R^{m+1}\subseteq R^{\leq m}$.  }
\end{equation}

Recall that
$\Di^{\leq m} \vf =
\bigvee_{i\leq m} \Di^i \vf$, $\Box^{\leq m} \vf =\neg \Di^{\leq m} \neg \vf$,
and $\DiAl\vf$ denotes $\bigvee_{\Di\in\Al}\Di \vf$.
For any $\Al$-frame $\frF$, we have the following equivalence \cite[Section 3.4]{KrachtBook}:
\begin{equation}\label{pretr:sem}
\text{
$R_\frF$ is $m$-transitive iff
$\frF\mo \DiAl^{m+1} p \imp \DiAl^{\leq m} p$. }
\end{equation}

A logic $\vL$ is said to be {\em $m$-transitive}, if
$\vL\vd \Di_\AlA^{m+1} p \imp \Di_\AlA^{\leq m} p$.
$\vL$ is {\em pretransitive} if it is $m$-transitive for some $m\geq 0$.


Put
$$
B_0^{[m]} =\bot,\quad
B^{[m]}_{h} = p_{h} \to  \Box_\AlA^{\leq m} (\Di_\AlA^{\leq m} p_{h} \lor B^{[m]}_{h-1}).
$$

Assume that $\frF=(X,(R_i)_{i<n})$ is $m$-transitive. In this case,  the modal operator $\Di_\AlA^{\leq m}$  relates to $R^*_\frF$.
Observe that
the height of $\frF$  is the height of the preorder $(X,R^*_\frF)$. Hence from
\eqref{eq:fin-ht-trans} we obtain:
\begin{equation}\label{eq:heigth-m}
\text{$\frF\mo B^{\leq m}_h$  iff  the height of $\frF$ is not greater than $h$.}
\end{equation}

 For a unimodal formula  $\vf$, let
 $\vf^{[m]}$ be the formula obtained from $\vf$ by replacing each occurrence of $\Di$ in $\vf$ with $\Di_\AlA^{\leq m}$. The following fact follows from \cite[Lemma 1.3.45]{ShefSkvGab}.

\begin{proposition}\label{prop:fragmS4}
For a pretransitive logic $\vL$, the set $\{\vf\mid \vL\vd \vf^{[m]} \}$ is a transitive logic.
\end{proposition}

Assume that $\vL$ is 1-finite. Then its 1-canonical frame is finite. Every finite frame is $m$-transitive for some $m$. Since $m$-transitivity is expressed by a formula in one variable,
by the Canonical Model Theorem (Proposition \ref{prop:k-canonical-model}) we obtain $\vL\vd \Di_\AlA^{m+1} p \imp \Di_\AlA^{\leq m} p$.\footnote{This observation is most probably a folklore; it is mentioned
in \cite[Chapter 7]{CZ} as an exercise.}

Moreover, if $\vL$ is 1-finite, then $\vL$ contains an axiom of finite height \cite{LocalTab16AiML}. Indeed, by Proposition \ref{prop:fragmS4}, the set $^{[m]}\vL=\{\vf\mid \vL\vd \vf^{[m]} \}$ is a transitive logic.
Since $\vL$ is 1-finite, the logic $^{[m]}\vL$ is also 1-finite.
The following result was obtained in \cite[Theorem 1]{Maksimova89}:
\begin{equation}\label{eq:max-on-1-fin}
\text{ A unimodal transitive logic is locally tabular iff
it is 1-finite. }
\end{equation}
Hence,  $^{[m]}\vL$ is locally tabular, and by \eqref{eq:thm-seg-maks},
$^{[m]}\vL\vd B_h$ for some $h<\omega$.
It follows that $\vL\vd B_h^{[m]}.$

So we have the following necessary condition of 1-finiteness:
\begin{theorem}\label{thm:1-finite-to-m-h}\cite{LocalTab16AiML}
If a logic $L$ is 1-finite, then
for some $h,m<\omega$, $L\vd \Di_\AlA^{m+1} p \imp \Di_\AlA^{\leq m} p$ and
$L\vd B_h^{[m]}.$
\end{theorem}
Unlike in the unimodal transitive case, in general the formulas in the above theorem do not imply local tabularity: there is a unimodal 2-transitive logic containing $B^{[2]}_1$ that
is not locally tabular \cite{Byrd78}, and even not 1-finite \cite{Makinson81}.

The equivalence \eqref{eq:max-on-1-fin} also does not transfer for the non-transitive case:
in Section \ref{subs:k-fin-and-subf}, we give examples of 1-finite but not locally tabular modal logics.

\subsection{Locally finite modal algebras and partitions of frames}


A {\em partition} $\clV$ of a set $X$ is a family
of non-empty pairwise disjoint sets such that $X=\bigcup \clV$.
Equivalently, a partition is the quotient set of $X$ by an equivalence; this equivalence is denoted by
$\sim_\clV$.
A partition $\clU$ {\em refines} $\clV$, if each element of $\clV$ is the union of some elements of $\clU$,  equivalently, $\sim_\clU\;\subseteq \;\sim_\clV$.
For a family $\clS$ of subsets of $X$, the {\em partition induced by} $\clS$
is the quotient set
$X/{\sim}$, where $$a\sim b \text{ iff }  \AA Y\in\clS \, (a\in Y \Leftrightarrow b\in Y).$$

\begin{definition}
\label{def:tune}
Let $X$ be a set, $R$ a binary relation on $X$.
A partition $\clU$ of $X$ is said to be {\em $R$-tuned}, if for every $U,V\in \clU$,
\begin{equation}\label{eq:part}
\EE a\in U  \EE b\in V \, (aR b)  \,\Rightarrow\, \AA a\in U  \EE b\in V \, (aR b).
\end{equation}
Let $\frF=(X,(R_\Di)_{\Di\in \Al})$ be an $\Al$-frame.
A partition $\clU$ of $X$ is said to be {\em tuned in $\frF$}, if it is $R_\Di$-tuned for every $\Di\in \Al$.
The frame $\frF$ is said to be {\em tunable}, if for
 every finite partition $\clV$ of $\frF$ there exists a finite tuned refinement $\clU$ of $\clV$.
A class $\clF$ of $\Al$-frames is called {\em $f$-tunable} for a function $f:\omega\to\omega$, if
 for every $\frF\in\clF$, for every finite partition $\clV$ of $F$ there exists a refinement
$\clU$ of $\clV$ such that $|\clU|\leq f(|\clV|)$ and $\clU$ is tuned in $\frF$.
A class $\clF$ is \em{uniformly tunable}, if it is $f$-tunable for some $f:\omega\to\omega$.
\end{definition}


In \cite{Franz-Bull},
tuned partitions were used in a proof of Bull's theorem (see the definition of
{\em Franzen’s  filtration} in \cite{Franz-Bull}).
In \cite[Corollary 3.3]{Blok1980}, it was observed that a tuned partition on a frame $F$ induces a subalgebra of the algebra
$\Alg{F}$; see Proposition \ref{prop:lfViaTuned-basic} below.

In \cite{LocalTab16AiML}, it was shown that
a modal logic is locally tabular iff it is the logic of a uniformly tunable class of frames; below we explain how this result follows from  Maltsev's criterion, Theorem \ref{thm:Maltsev}.

It is immediate from the above definition that
\begin{equation}\label{eq:tuned-alternative}
\textstyle \text{$\clU$ is $R$-tuned} \text{ iff } \AA V\in\clU\,\EE \clS\subseteq \clU \; (R^{-1}[V]=\bigcup \clS).
\end{equation}

\begin{proposition}\label{prop:lfViaTuned-basic}
Let $\clU$ be a partition of a frame $\frF$. Then $\clU$ is tuned in $\frF$ iff the family
$\{\bigcup \clS\mid \clS\subseteq \clU\}$ forms a subalgebra of the modal algebra $\Alg{\frF}$.
\end{proposition}
\begin{proof}
Let $\clT=\{\bigcup \clS\mid \clS\subseteq \clU\}$,  $R$ one of the relations of $F$.

Assume that $\clU$ is tuned in $F$. Clearly,
$\clT$ is closed under the Boolean operations.
For $\clS \subseteq \clU$, we have $R^{-1}[\bigcup \clS]=\bigcup \{R^{-1}[V]\mid V\in \clS\}$, so it follows from  \eqref{eq:tuned-alternative} that $R^{-1}[\bigcup \clS]\in\clT$. Hence,
$\clT$  is closed under operations of $\Alg{F}$.

Conversely, assume that $\clT$  forms a subalgebra of $\Alg{F}$. Then for every $V\in \clU$,
$R^{-1}[V]\in\clT$.  Hence, $R^{-1}[V]$ is the union of some elements of $\clU$. By  \eqref{eq:tuned-alternative}, $\clU$ is $R$-tuned.
\end{proof}

\begin{proposition}\label{prop:betweenSizeOfAlgAndPart}
Let $k<\omega$, $\clV$ the partition induced by subsets $P_0,\ldots, P_{k-1}$ of a frame $F$,
$B$ the subalgebra of $\Alg{F}$ generated by these sets.
\begin{enumerate}[\normalfont (1)]
  \item \label{prop:SizeOfAlg-To-Part} If $B$ is finite, then the partition $\clU$ of $F$ induced by the elements of $B$ is a tuned refinement of $\clV$.
  \item \label{prop:SizePartToOfAlg}
If $\clU$ is a tuned refinement of $\clV$, then  every element of $B$ is the union of some elements of $\clU$.
\end{enumerate}
\end{proposition}
\begin{proof}
\eqref{prop:SizeOfAlg-To-Part}
All sets $P_i$ belong to $B$,  so $\sim_\clU\;\subseteq \;\sim_\clV$, and so $\clU$ refines $\clV$.
Since $B$ is finite, the partition $\clU$ is the set of atoms of $B$ (here $B$ considered as a Boolean algebra), and the domain of $B$ is the family $\{\bigcup \clS\mid \clS\subseteq \clU\}$. By Proposition \ref{prop:lfViaTuned-basic}, $\clU$ is tuned.

\eqref{prop:SizePartToOfAlg} By Proposition \ref{prop:lfViaTuned-basic},
the family $\{\bigcup \clS\mid \clS\subseteq \clU\}$ forms
a subalgebra $C$ of $\Alg{\frF}$.
Since $\clU$ refines $\clV$, all $P_i$ are in $C$, and hence
$C$ contains $B$.
\end{proof}

Let $k<\omega$.
For a $k$-valuation $\theta$ in an $\Al$-frame $F$, let
$\sim_\theta$ be the equivalence on $\dom{F}$ induced by all $k$-formulas
in the model $(F,\theta)$. In terms of Kripke models, $\sim_\theta$ is the set of pairs $(a,b)\in \dom{F}\times \dom{F}$ such that
for
every $k$-formula  $\vf$,
$$
(F,\theta), a\mo \vf \Leftrightarrow (F,\theta), b\mo \vf.
$$

\begin{proposition}\label{prop:partIn-k-model}
If the algebra $\Alg{F}$ is locally finite, then
for every $k<\omega$ and every $k$-valuation $\theta$ in $F$, the partition $\dom{F}{/}{\sim_\theta}$ is finite and tuned in $F$.
\end{proposition}
\begin{proof}
Consider the algebra $B$ generated in $\Alg{F}$  by the sets $\theta(p_0), \ldots, \theta(p_{k-1})$.
Then $\dom{F}{/}{\sim_\theta}$ is the partition of $F$ induced by the elements of $B$.  Now the statement follows from
Proposition \ref{prop:betweenSizeOfAlgAndPart}\eqref{prop:SizeOfAlg-To-Part}.
\end{proof}

\begin{theorem}\cite{LocalTab16AiML}\label{thm:LFviaTuned}
~\begin{enumerate}[\normalfont (1)]
  \item  \label{thm:LFviaTuned-lfA}
The algebra of a frame $\frF$ is locally finite iff $\frF$ is tunable.
\item \label{thm:LFviaTuned-lf-k-fin} Let $k<\omega$. The logic of a class  $\clF$ of $\Al$-frames is $k$-finite
iff there exists $n<\omega$ such that
for every $\frF\in\clF$, for every partition $\clV$ of $F$ with $|\clV|\leq 2^k$
there exists a finite tuned refinement $\clU$ of $\clV$ with $|\clU|\leq n$.
  \item \label{thm:LFviaTuned-lfLog}
The logic of a class  $\clF$ of $\Al$-frames is locally tabular
iff $\clF$  is uniformly tunable.
\end{enumerate}
\end{theorem}
\begin{proof}
\eqref{thm:LFviaTuned-lfA} The `only if' direction follows from Proposition \ref{prop:betweenSizeOfAlgAndPart}\eqref{prop:SizeOfAlg-To-Part}, the `if' direction from
Proposition \ref{prop:betweenSizeOfAlgAndPart}\eqref{prop:SizePartToOfAlg}.

\eqref{thm:LFviaTuned-lf-k-fin}
By Maltsev's criterion provided in Theorem \ref{thm:Maltsev}\eqref{thm:Maltsev-kfin},
$\Log{\clF}$ is $k$-finite iff
there exists $l<\omega$ such that
for every $\frF\in \clF$, the cardinality of every subalgebra of $\Alg{\frF}$ generated by $k$ subsets of $F$
in not greater than $l$.

Assume that $\Log{\clF}$ is $k$-finite. Let $\clV$ be a partition of $F\in\clF$
such that $|\clV|\leq 2^k$. Then there are $k$ subsets $P_0,\ldots, P_{k-1}$ that induce $\clV$
(consider an injection $f:\clV\to\clP(k)$, and put $P_i=\bigcup\{V\in\clV\mid i\in f(V)\}$). Let $B$ be the subalgebra of $\Alg{F}$ generated by these sets,
and let $\clU$ be
the partition of $F$ induced by the elements of $B$.
By Proposition \ref{prop:betweenSizeOfAlgAndPart}\eqref{prop:SizeOfAlg-To-Part}, $\clU$ is a tuned refinement of $\clV$. Since $B$ is finite, $|B|=2^{|\clU|}$.  So if $|B|\leq l$, then $2^{|\clU|}\leq l$.

For the other direction, consider a frame $F\in\clF$ and a
subalgebra $B$ of $\Alg{F}$ generated by subsets $P_0,\ldots, P_{k-1}$ of $F$.
Let $\clV$ be the partition induced by these sets; clearly, $|\clV|\leq 2^k$.
For some fixed $n$,  $\clV$ has a finite tuned refinement $\clU$ of size not greater than $n$.
By Proposition \ref{prop:betweenSizeOfAlgAndPart}\eqref{prop:SizePartToOfAlg},
$|B|\leq 2^n$. This completes the proof of the second statement.

The third statement readily follows from the second.
%
%
%
\end{proof}

In \cite{LocalTab16AiML}, this theorem was used to generalize the finite height criterion \eqref{eq:thm-seg-maks}
for certain  families of non-transitive logics.

Note that the first statement of this theorem
does not imply that $\Log(F)$ is locally tabular. However, Theorem \ref{thm:LFviaTuned}\eqref{thm:LFviaTuned-lfA}  can be a useful tool for proving the finite model property of a semantically defined logic.
For example, consider the logic of a direct power $F=(\omega,\leq)^n$ of
$(\omega,\leq)$, where $0<n<\omega$. 
This logic is not locally tabular, since $F$ is not of finite height. On the other hand,
$F$ is tunable \cite{OmegaN}, so the algebra of $F$  is locally finite, and so the logic
of $F$ has the finite model property.

\section{Local tabularity and subframes}\label{sec:subframes}

For a class $\clF$ of $\Al$-frames, put
  $$\SubFrs{\clF}=\{\frF\restr Y\mid \frF\in \clF \text{ and } Y\subseteq \dom{\frF}\}.$$

\subsection{k-finiteness and subframes}\label{subs:k-fin-and-subf}

\begin{proposition}\label{prop:LF-for-subframess}
If the logic of a class $\clF$ of $\Al$-frames is $k$-finite for some positive $k<\omega$,
then
the logic of $\SubFrs{\clF}$ is ${(k{-}1)}$-finite.
\end{proposition}
\begin{proof}
Via Theorem \ref{thm:LFviaTuned}.
Let $Y$ be a subset of a frame $F\in \clF$, and $\clV$ be a finite partition of $Y$, $|\clV|\leq 2^{k-1}$.
Consider the partition $\clV'=\clV\cup\{\dom{F}{\setminus} Y\}$ of $F$ (if $Y=\dom{F}$, put $\clV'=\clV$).
Then $\clV'\leq 2^k$, and since $\Log{\clF}$ is $k$-finite,  by Theorem \ref{thm:LFviaTuned}\eqref{thm:LFviaTuned-lf-k-fin}
there exists a finite tuned refinement $\clU$ of $\clV'$
with $|\clU|\leq n$, for some fixed finite $n$. Consider the partition $\clU_\prime=\{V\in\clU\mid V\subseteq Y\}$ of $Y$. It is straightforward that $\clU_\prime$ refines $\clV$, is tuned
in $F\restr Y$, and $|\clU_\prime|\leq n$. Using
Theorem \ref{thm:LFviaTuned}\eqref{thm:LFviaTuned-lf-k-fin}  again, we conclude that $\Log{\SubFrs{\clF}}$ is ${(k{-}1)}$-finite.
\end{proof}

\begin{corollary}
The logic of $\clF$ is locally tabular iff
  the logic of $\SubFrs{\clF}$ is.
\end{corollary}
\begin{proof}
The ``if'' direction is immediate:  $\clF\subseteq \SubFrs{\clF}$, so $\Log{\SubFrs{\clF}}\subseteq \Log{\clF}$.
The ``only if'' direction follows from Proposition \ref{prop:LF-for-subframess}.
\end{proof}

An analogous statement holds for algebras. Namely, if the algebra of a frame $\frF$ is locally finite, then for
every $Y\subseteq \dom{\frF}$, the algebra of the frame $\frF\restr Y$
 is locally finite
(same argument as in the proof of Proposition \ref{prop:LF-for-subframess}, via Theorem \ref{thm:LFviaTuned}\eqref{thm:LFviaTuned-lfA}). In particular,
 it follows that all logics $\Log{\frF\restr Y}$ have the finite model property.

\begin{proposition}\label{prop:2finite-to-subs-m-n}
If the logic of a class $\clF$ of frames is 2-finite,
then there exist $m,h<\omega$ such that for every $F\in \SubFrs{\clF}$, the frame $F$ is $m$-transitive, and the height of $F$ is not greater than $h$.
\end{proposition}
\begin{proof}
If $\Log{\clF}$  is 2-finite, then logic $\Log{\SubFrs{\clF}}$ is 1-finite by Proposition
\ref{prop:LF-for-subframess}.
Now the statement follows from Theorem  \ref{thm:1-finite-to-m-h} and equivalences \eqref{pretr:sem} and \eqref{eq:heigth-m}.
\end{proof}

While in the unimodal transitive case, 1-finiteness of a modal logic implies local tabularity \eqref{eq:max-on-1-fin},
this is not true in general. The followings example was given in \cite{Glivenko2021}.
Let $\vL$ be the logic of the frame
$\frF=(\omega+1,R)$, where
$$aRb \tiff  a\leq b  \textrm{ or }a=\omega.$$
The restriction of $\frF$ to $\omega$ is the frame $(\omega,\leq)$, so $\vL$ is not 2-finite by Proposition \ref{prop:2finite-to-subs-m-n}.
On the other hand,
for every 2-element partition $\clV$ of $F$,
there exists its tuned refinement $\clU$ with $|\clU|\leq3$. Indeed, let $\clV=\{V_0,V_1\}$.
W.l.o.g., let $V_0\subseteq \omega$. It is immediate that if $V_0$ is infinite, then $\clV$ is tuned in $\frF$.
Otherwise, let $n$ be the greatest element of $V_0$,
$U_1=\{x\mid n<x<\omega\}$, $U_2$ the complement of $V_0\cup U_1$ in $F$. It is straightforward from Definition \ref{def:tune} that $\{V_0,U_1,U_2\}$ is a tuned refinement of $\clV$.
By Theorem
\ref{thm:LFviaTuned}\eqref{thm:LFviaTuned-lf-k-fin}, the logic of $F$ is 1-finite.

\hide{
This proves
\begin{theorem}\cite{Glivenko2021}
There exists a unimodal 1-finite logic which is not locally tabular.
\end{theorem}
}

In fact, there are 1-finite but not locally tabular logics above the logic $\LK{TB}$ of reflexive and symmetric relations.  Namely,  consider the logic $L$ of the following, reflexive and symmetric, structure  $F=(\mathbb{Z},R)$:
$$aRb \tiff |a-b|\neq 1.$$
Let $F_0$ be the restriction of $F$ on the set of natural numbers. Then the subalgebra of $\Alg(F_0)$ generated by the singleton $\{0\}$ contains all singletons $\{a\}$, $a\in \omega$ (simple induction on $a$). Hence, $\Alg(F_0)$ is not 1-finite. By Proposition \ref{prop:LF-for-subframess},
$L$ is not 2-finite.
On the other hand, we claim that $L$ is 1-finite. To show it, consider
a 2-element partition $\clV=\{V_0,V_1\}$ of $\mathbb{Z}$, and let $|V_0|\leq |V_1|$.
It is easy to verify that if $a,b\in V_0$ for some distinct $a,b\in \mathbb{Z}$ with $|a-b|\neq 2$, then $\clV$ is tuned in $F$.
In the remaining cases, $|V_0|\leq 2$, and we have $V_0=\{a\}$ or $V_0=\{a-1,a+1\}$ for some $a\in \mathbb{Z}$.
In both cases,
the three element partition $\clU$ that contains $\{a\}$ and  $\{a-1,a+1\}$ is a tuned refinement of $\clV$.
By Theorem
\ref{thm:LFviaTuned}\eqref{thm:LFviaTuned-lf-k-fin}, $L$ is 1-finite.
This proves
\begin{theorem}\label{thm:1-not-k-abobe-KTB}
There exists a unimodal 1-finite logic above $\LK{TB}$, which is not locally tabular.
\end{theorem}

Hence, in the non-transitive case, 1-finiteness does not imply local tabularity.
We do not know if $k$-finiteness of a modal logic implies its local tabularity for some fixed finite $k$.\footnote{In recent manuscripts \cite{Hk-1} and \cite{Hk-2}, it is claimed that no such $k$ exists for intermediate logics.}
\hide{
\begin{problem}
Does there exist a finite $k$ such that $k$-finiteness of a modal logic implies its local tabularity?
\end{problem}
The same question is open in the case of intermediate logics \cite[Problem 2.4]{GuramRevazProblem}.}

\subsection{Relativized height and relativized pretransitivity}
Let $M=(F,\theta)$ be a $k$-model, $V\subseteq \dom F$.
The {\em restriction $M\restr V$ of $M$ to $V$} is the $k$-model $(F\restr V,\theta\restr V)$, where
$\theta\restr V(p_i)=\theta(p_i)\cap V$ for all $i<k$.

The following relativization argument was proposed in  \cite{Spaan1993complexity}.\footnote{ See the proof of Theorem 2.2.1 in \cite{Spaan1993complexity}.}
Let $k\leq \omega$, $\xi$ a $k$-formula. By induction on the structure of a $k$-formula $\vf$,
we define $\trSub{\vf}{\xi}$:
$\trSub{\bot}{\xi}=\bot$, $\trSub{p}{\xi}=p$ for variables, $\trSub{\psi_1\imp \psi_2}{\xi}=\trSub{\psi_1}{\xi}\imp \trSub{\psi_2}{\xi}$,
and $\trSub{\Di \psi}{\xi}= \Di (\xi\con \trSub{\psi}{\xi})$ for $\Di\in\Al$.

\begin{proposition}\label{prop:relSpaan}
Let $M=(F,\theta)$ be a $k$-model, $\xi$ a $k$-formula, and $V=\vext(\xi)$. Then for every $a\in V$,
for every $k$-formula $\vf$,
$$\mM\restr V,a\mo \vf \text{ iff } \mM,a\mo \trSub{\vf}{\xi}$$
\end{proposition}
\begin{proof}
By induction on the structure of $\vf$.
\end{proof}
\begin{proposition}\label{prop:relSpaanVal}
Let $\clF$ be a class of $\Al$-frames, $\vf$ a formula in the language of $\Al$,  and $q$ a variable not occurring in $\vf$. Then
$$
\clF\mo q\imp \trSub{\vf}{q} \text{ iff }\SubFrs{\clF}\mo\vf.
$$
\end{proposition}
\begin{proof}
Follows from Proposition \ref{prop:relSpaan}.
\end{proof}

For a binary relation $R$ on a set $X$ and $V\subseteq X$, let $R_V$ denote the restriction of $R$ to $V$,
and let $R^*_V$  be
the transitive reflexive closure of $R_V$, that is:
$$
 R^*_V=\{(a,a)\mid a\in V\}\cup (R_V) \cup (R_V)^2  \cup \ldots
$$
A formula is said to be {\em Boolean}, if it is build using only Boolean connectives.
\begin{proposition}\label{prop:relativised-m-trans}
Assume that the logic of an $(l+1)$-frame $F=(X,R,R_1,\ldots,R_l)$ is $2$-finite.
Then there is $m<\omega$ such that for every
$k$-model $M=(F,\v)$, for every $k$-formula $\xi$, for every Boolean $k$-formula $\vf$,
for all $a\in V=\bar{\theta}(\xi)$, we have:
\begin{eqnarray}
  &&M,a\mo \trSub{\Box_0^{\leq m} \vf}{\xi} \text{ iff } R_{V}^*(a)\subseteq \vext(\vf). \label{eq:pretrInSub-Box}
\end{eqnarray}
\end{proposition}
\begin{proof}
Observe that $\trSub{\vf}{\xi}$ is $\vf$. By a straightforward induction on $i$ we obtain:
$$M,a\mo \trSub{\Di_0^{i}\vf}{\xi}\text{ iff } \EE b\;( (a,b)\in R_{V}^i\,\&\,M,b\mo \vf ).$$
It follows that
$$M,a\mo \trSub{\Di_0^{\leq i}\vf}{\xi}\text{ iff } \EE b\;( (a,b)\in R_{V}^{\leq i}\,\&\,M,b\mo \vf ).$$
Since the logic of $F$ is 2-finite, the logic of the frame $(X,R)$ is 2-finite as well.  By Proposition \ref{prop:2finite-to-subs-m-n}, there is $m <\omega$ such that for every $Y\subseteq X$, the frame $(X,R)\restr Y$ is $m$-transitive, and so $(R\restr V)^{\leq m}=R^*_V$. Hence we have:
$$M,a\mo \trSub{\Di_0^{\leq m} \vf}{\xi}\text{ iff } \EE b\;( (a,b)\in R_{V}^*\,\&\, M,b\mo \vf ),$$
and so \eqref{eq:pretrInSub-Box} follows.
\end{proof}

We conclude this section with a two-variable analog of Theorem \ref{thm:1-finite-to-m-h}.

\begin{theorem}
Assume that a logic $L$ is $2$-finite. Then there exists $m<\omega$ such that:
\begin{enumerate}[\normalfont (1)]
\item
$L\vd q\imp \trSub{\Di_\AlA^{m+1}p\imp \Di_\AlA^{\leq m} p}{q}$, and
\item if also $L$ is Kripke complete, then there exists $h$ such that
$L\vd q\imp \trSub{B^{[m]}_{h}}{q}$.
\end{enumerate}
\end{theorem}
\begin{proof}
The first statement follows from the same argument as in Theorem \ref{thm:1-finite-to-m-h}.
Namely, the 2-canonical frame $\frF$ of $L$ is finite.
The case when $L$ is inconsistent is trivial. Otherwise, $\frF$ is non-empty.  Put $m=|\dom{F}|-1$. Then every
subframe of $\frF$ is $m$-transitive, and so $\Di_\AlA^{m+1}p\imp \Di_\AlA^{\leq m}$ is valid there
by \eqref{pretr:sem}.
By Proposition \ref{prop:relSpaanVal},
$\vf=q\imp\trSub{\Di_\AlA^{m+1}p\imp \Di_\AlA^{\leq m} p}{q}$ is valid in $\frF$.
Since $\vf$ is a formula in two variables,  $L\vd \vf$ by the Canonical Model Theorem (Proposition \ref{prop:k-canonical-model}).

Now assume that $L$ is Kripke complete and that $\clF$ is the class of $L$-frames.
By Proposition \ref{prop:2finite-to-subs-m-n},
there exist $h,m<\omega$ such that
the formulas  $\Di_\AlA^{m+1} p \imp \Di_\AlA^{\leq m} p$ and $B_h^{[m]}$
are valid in $\SubFrs{\clF}$, and so the corresponding
relativizations are valid in $\clF$ by Proposition \ref{prop:relSpaanVal}.
\end{proof}

We conjecture that the second statement remains true for Kripke incomplete logics.

\section{Local tabularity and reflexivity}\label{sec:refl}

For a frame $F=(X,(R_\Di)_{\Di\in \Al})$, let $\refl{F}$ be the
frame $(X,(\refl{R_\Di})_{\Di\in \Al})$, where $\refl{R_\Di}$ is the reflexive closure
$R_\Di\cup\{(a,a)\mid a\in X\}$ of $R_\Di$.
For a class $\clF$ of frames, $\refl{\clF}=\{\refl{F}\mid F\in\clF\}$.

We will show that $\Log{\clF}$ is locally tabular iff $\Log{\refl{\clF}}$ is.
Once the `only if' direction is straightforward, the `if' direction is non-trivial and is based on the following crucial lemma.

\begin{lemma}\label{lem:refl-reduction-key-uni}
Let $F=(X,R,R_1,\ldots, R_{l})$ be an $(l{+}1)$-frame,
$G=(X,\refl{R},R_1,\ldots R_{l})$, and $\Log{G}$ locally tabular.
Then for every $k<\omega$, every $k$-generated subalgebra of $\Alg{F}$
is contained in a $(k+3)$-generated subalgebra of $\Alg{G}$.
\end{lemma}

\begin{proof}

Assume that a subalgebra $A$ of $\Alg{F}$ is generated by $P_0,\ldots P_{k-1}\subseteq X$.
We will construct three sets $$Q,E,S\subseteq X$$
such that $A$ is contained
in the subalgebra of $\Alg{G}$ generated by $P_0,\ldots P_{k-1}, Q,E,S$.
This construction requires a number of technical steps, and we will split it into the following parts:
definition of $Q$;  auxiliary constructions; definitions of $E$ and $S$; embedding the algebras.

\bigskip

\subparagraph{{\bf Defining $Q$.}}

Let $U\subseteq X$, $c\in U$.  We say that $c$ is a {\em defect of} $U$, if
$aRb$ for some  $a,b\in U$, and
$cRd$ for no $d$ in $U$.  The set of all defects of $U$ is denoted
by $\Def{U}$. A set is {\em defective}, if it contains a defect.


\begin{claim}\label{claim:noDef-tuned}
 Let ${\sim}$ be an equivalence on $X$ such that
$X{/}{\sim}$ is tuned in $G$. If there are no defective $\sim$-classes,
then $X{/}{\sim}$ is also tuned in $F$.
\end{claim}
\begin{proof}
We only need to check that $X{/}{\sim}$ is $R$-tuned.
Assume that $U,V\in X{/}{\sim}$, $a R b$ for some $a\in U$ and $b\in V$. Let $a'\in U$.
We need to show that $a'R b'$ for some $b'\in V$.

Assume that $U\neq V$.
Since $X{/}{\sim}$ is tuned in $G$, there is $b'\in V$ such that $a'\refl{R}b'$.
Since $a'\neq b'$, we get $a'Rb'$.

If $U=V$, then, since $aRb$, $a,b\in U$, and $U$ is not defective, we have $a'Rb'$ for some $b'\in U$.
\end{proof}

Consider the $k$-valuation $\theta:\{p_0,\ldots,p_{k-1}\}\to \clP(X)$, $\theta(p_i)=P_i$ for $i<k$.

By recursion,
we define a sequence of equivalences on $X$
\begin{equation}\label{eq:simn-refine}
\sim_0\;\supseteq\;\sim_1\;\supseteq\;\sim_2\;\supseteq\;\ldots,
\end{equation}
and subsets $Q_0,Q_1,Q_2\ldots$  of $X$.

Let $\sim_0$ be the equivalence induced by all formulas
in variables $p_0,\ldots,p_{k-1}$ in the model $(G,\theta)$. Let $Q_0$ be the set $\bigcup\{\Def{U}\mid U\in X{/}{\sim_0}\}$ of defects occurring in $\sim_0$-classes.

To define $\sim_{n+1}$ and $Q_{n+1}$,  we consider
the $(k{+}n{+}1)$-model $(G,\theta_n)$ such that {$\theta_n: \{p_0,\ldots,p_{k-1},q_0,\ldots,q_{n}\}\to \clP(X)$} extends
$\theta$ and $\theta_n(q_0)=Q_0$, \ldots, $\theta_n(q_n)=Q_n$. Let $\sim_{n+1}$ be
the
equivalence   induced
by all formulas in variables $p_0,\ldots,p_{k-1},q_0,\ldots,q_{n}$ in the model $(G,\theta_n)$.
Let $Q_{n+1}=\bigcup\{\Def{U}\mid U\in X{/}{\sim_{n+1}}\}$.

We put $$Q=\bigcup_{n<\omega} Q_n.$$

\bigskip

\subparagraph{{\bf Preliminaries for defining $E$ and $S$.}}
The role of $E$ and $S$ is to ``distinguish'' subsets $Q_n$ of $Q$.
To define $E$ and $S$, we need some preliminary observations about the constructed partitions.

\begin{claim}\label{claim:simi-are-finandtuned}
All quotient sets $X{/}{\sim_n}$ are finite and tuned in $G$.
\end{claim}
\begin{proof}
Since the logic of the frame $G$ is locally tabular, its algebra is locally finite.
Now the claim follows
from Proposition \ref{prop:partIn-k-model}.
\end{proof}

For $a\in X$, let $[a]_n$ denote the $\sim_n$-class of $a$.
\begin{claim}\label{claim:detachDefects}
If $U\in X{/}{\sim_{n}}$, $a\in \Def{U}$ and $m>n$, then  $[a]_{m}\subseteq \Def U$ and
$[a]_{m}$ is not defective.
\end{claim}
\begin{proof}
We have $a\in Q_n\cap U$. By the definition of $\sim_{n+1}$, $[a]_{n+1}\subseteq Q_{n}\cap U= \Def U$.
Since $m\geq n+1$, $[a]_m\subseteq [a]_{n+1}$ by \eqref{eq:simn-refine}.
So $[a]_m$ is contained in $\Def{U}$.
It follows that
$bRc$  for no $b,c\in [a]_m$, and so $[a]_m$ is not defective.
\end{proof}

\begin{claim}\label{claim:Qndisjoint}
All sets $Q_n$ are disjoint.
\end{claim}
\begin{proof}
Assume $a\in Q_n$ for some $n$. Then $a\in \Def{U}$, where $U=[a]_n$.
For every $m>n$, the $\sim_{m}$-class of $a$ is not defective by Claim \ref{claim:detachDefects}. Hence, $a\notin Q_m$.
\end{proof}

By Claim \ref{claim:Qndisjoint},
for any $a\in Q$ there exists a unique $n$,
the {\em index of  $a$},
 such that $a\in Q_n$; we denote such an $n$ by $\ind{a}$.

\hide{
\begin{claim}\label{claim:index}
If $a\in Q$, then $[a]_n$ is not defective for every $n>\ind{a}$.
\end{claim}
\begin{proof}
By the definition, $a$ is a defect of the set $[a]_{\ind{a}}$. By Claim \ref{claim:detachDefects},
$[a]_n \subseteq \Def{[a]_{\ind{a}}}$.
By the definition of a defect, $R\restr \Def{[a]_{\ind{a}}} =\emp$, so
$R\restr [a]_n =\emp$. Hence,  $[a]_n$ is not defective.
\end{proof}
}

\begin{claim}\label{claim:ind}
Let $U\in X{/}{\sim_{n}}$, $a\in U\cap Q$, and $\ind{a}> n$. Then $aRc$ for some $c\in U$.
\end{claim}
\begin{proof}
$a$ is a defect of some $V\in X{/}{\sim_m}$ with $m> n$. Then $V\subseteq U$; by the definition of a defective set, $b R c$ for some $b,c\in V$.
So $b,c \in U$, and since $a$ is not a defect of $U$,
$aRd$ for some $d$ in $U$.
\end{proof}

Let $\AllPart$ be the union of all quotients $X{/}{\sim_n}$:
$$\AllPart=\bigcup\{X{/}{\sim_n}\mid n<\omega\}.$$
Let $\AllDef$ denote the set of all defective sets in $\AllPart$:
$$
\AllDef=\{V\in \AllPart\mid \Def{V}\neq \emp\}.
$$
Let $V\in \AllDef$.  It is immediate that
there exists a unique $n$ such that $V\in X{/}{\sim_n}$;
this $n$ is called the {\em index of $V$}, in symbols $\ind{V}$.
It is also immediate that if $a\in \Def{V}$, then $\ind{a}=\ind{V}$.

\begin{claim}\label{claim:ind-a-greater} If $a\in V\cap Q$ for some $V\in \AllDef$, then
$\ind{a}\geq \ind{V}$.
\end{claim}
\begin{proof}
If $m>\ind{a}$, then $[a]_m$ is not defective by Claim \ref{claim:detachDefects}.
Since $V$ is defective, $V=[a]_m$ for some $m\leq \ind{a}$.
\end{proof}

\bigskip

\subparagraph{{\bf Defining $E$.}}

For $U\in \AllPart$, put
 $$\Sigma(U)=\{V\in\AllPart \mid U\subset V\}.$$
Clearly, for all $U,U'$ in $\AllPart$, we have
\begin{equation}\label{eq:AllPart}
U\subseteq U'\text{ or } U'\subseteq U \text{ or } U'\cap U=\emp.
\end{equation}
Hence, $\Sigma(U)$ is a $\subseteq$-chain.
We put $$
{\Sigma}^\up(U)=\{V\in \Sigma(U)\mid \EE a \in U \, \EE b\in \Def{V} \;aRb\}.
$$
\hide{
$\EE a \in U \, \EE b\in \Def{V}$ does not let us prove Same Index claim; update: we do not need it anymore
}

\begin{claim}\label{claim:sigma-forall}
For $U\in \AllPart$, ${\Sigma}^\up(U)=\{V\in \Sigma(U)\mid \AA a \in U \, \EE b\in \Def{V} \;aRb\}.$
\end{claim}
\begin{proof}
Let us check that ${\Sigma}^\up(U)\subseteq \{V\in \Sigma(U)\mid \AA a \in U \, \EE b\in \Def{V} \;aRb\}.$
Assume that  $V\in {\Sigma}^\up(U)$. We have $U\in X{/}{\sim_n}$ for some $n$. Then
$U\subset V\in X{/}{\sim_i}$ for some $i<n$, and we have $aRb$  for some
$a \in U$, $b\in \Def{V}$. We have $U=[a]_n\subseteq [a]_{i+1}$.
Let $a'\in U$.
Since $X{/}{\sim_{i+1}}$ is tuned in $G$, there is $b'\in [b]_{i+1}$ with $a'\refl{R}b'$.
By Claim \ref{claim:detachDefects},  $b'\in \Def{V}$. It remains to observe that
$a'\neq b'$: indeed, since $a'\sim_{i+1} a$ and $a$ is not a defect of $V$,
by Claim \ref{claim:detachDefects} we conclude
that $a'$ is not a defect of $V$ as well. Hence, $a'Rb'$.

The other inclusion is trivial, since $U$ is non-empty.
\end{proof}

\begin{claim}\label{claim:rks-monot}
If $U, U'\in \AllPart$, $U'\sub U$, then ${\Sigma}^\up(U')\supseteq {\Sigma}^\up(U)$.
\end{claim}
\begin{proof}
Immediate from Claim \ref{claim:sigma-forall}.
\end{proof}

The cardinality of ${\Sigma}^\up(U)$ is called the {\em rank of} $U$ and is denoted by $\rk{U}$.
\begin{claim}\label{claim:rks-all-above}
Assume that for $U\in\AllPart$, ${\Sigma}^\up(U)=\{V_{\rk{U}-1},\ldots, V_0\}$ with $V_{\rk{U}-1}\subset \ldots \subset V_0$. Then for each $j<\rk{U}$ we have
${\Sigma}^\up(V_j)=\{V_{j-1},\ldots, V_0\}$ and so
$\rk{V_j}=j$.
\end{claim}
\begin{proof} Take $j<\rk{U}$.
By Claim \ref{claim:rks-monot}, ${\Sigma}^\up(V_j) \subseteq \{V_{\rk{U}-1},\ldots, V_0\}$. Since ${\Sigma}^\up(V_j)\subseteq \Sigma(V_j)$, $V_i \in {\Sigma}^\up(V_j)$ only if $i<j$.
So ${\Sigma}^\up(V_j) \subseteq \{V_{j-1},\ldots, V_0\}$.
If $i<j$, then $V_i\in {\Sigma}^\up(V_j)$ by
the definition of ${\Sigma}^\up$, since $U\subset V_j\subset V_i$ and $V_i\in{\Sigma}^\up(U)$.
\end{proof}

For $a\in Q$, the {\em rank of $a$}, $\rk{a}$, is defined as the rank of
the $\sim_{\ind{a}}$-class of $a$:
$$
\rk{a}=\rk{([a]_{\ind{a}})}.
$$

\begin{claim}\label{claim:ranks}
Let $V\in \AllDef$, $a \in Q\cap V$.
Then:
\begin{enumerate}[\normalfont(i)]
  \item $\rk{a}\geq \rk{V}$; \label{claim:ranks-1}
  \item If $\rk{a}> \rk{V}$, then there exists $b\in V\cap Q$
 such that $aRb$ and $\rk{b}=\rk{a}-1$. \label{claim:ranks-2}
\end{enumerate}
\end{claim}
\begin{proof}
Let $n=\ind{V}$, that is $V\in X{/}{\sim_n}$.
Since $V=[a]_n$ is defective, $n\leq\ind{a}$
by Claim \ref{claim:ind-a-greater}.
Let $V'=[a]_{\ind{a}}$. Is follows that $V'\subseteq V$.
By Claim \ref{claim:rks-monot}, we have  ${\Sigma}^\up(V')\supseteq {\Sigma}^\up(V)$.
Now the first statement follows.

Let $\rk{a}> \rk{V}$. In this case ${\Sigma}^\up(V')\supset {\Sigma}^\up(V)$. So we have for some
$V_{\rk{a}-1}\subset \ldots \subset V_0$:
\begin{equation}\label{eq:sigmupVVprime}
  {\Sigma}^\up(V')=\{V_{\rk{a}-1},\ldots, V_{\rk{V}-1},\ldots, V_0\},\quad
  {\Sigma}^\up(V)=\{V_{\rk{V}-1},\ldots, V_0\}.
\end{equation}
By Claim \ref{claim:sigma-forall},
there exists $b\in \Def{V_{\rk{a}-1}}$ with $aRb$.
We have $V_{\rk{a}-1}=[b]_{\ind{b}}$, and so $\rk{b}=\rk{V_{\rk{a}-1}}$, and hence $\rk{b}=\rk{a}-1$ by Claim \ref{claim:rks-all-above}.

It remains to show that $b\in V$.
By \eqref{eq:sigmupVVprime},  $V_{\rk{a}-1}\notin {\Sigma}^\up(V)$.
It follows that $V_{\rk{a}-1}\notin \Sigma(V)$, since we have $a \in V$, $aRb$, $b\in \Def{V_{\rk{a}-1}}$.
Hence,  $V$ is not a proper subset of $V_{\rk{a}-1}$. Since $a\in V\cap V_{\rk{a}-1}$,
by \eqref{eq:AllPart}
we have $V_{\rk{a}-1}\subseteq V$. Hence $b\in V$.
\end{proof}

Let $E$ be the set of $a\in Q$ of even rank:
$$
E=\{a\in Q\mid \rk{a} \text{ is even}\}.
$$

\bigskip

\subparagraph{{\bf Defining $S$.}}

For $V\in\AllDef$, put
\begin{eqnarray}
  \overline{V}&=&(V\setminus Q_{\ind{V}})\cap Q,\\
G_V&=&(\overline{V},(R\restr \overline{V})^*).
\end{eqnarray}
Let $\Cl{V}$ be the set of all maximal clusters in the preordered set $G_V$,
\begin{equation*}
\AllCl=\{C\in \bigcup\{\Cl{V}\mid V\in \AllDef\}\mid |C|>1\}.
\end{equation*}

\begin{claim}\label{claim:DCdisjoint}~
\begin{enumerate}[\normalfont(i)]
\item If $C\in\Cl{V}$, $C'\in\Cl{V'}$ for some $V,V'\in \AllDef$,  $C\cap C'\neq \emp$, and
 $\ind{V'}\leq \ind{V}$, then $C\subseteq C'$.\label{AllC-ind-monot}
\item For all $C,C'\in \AllCl$, we have $C\subseteq C'\text{ or } C'\subseteq C \text{ or } C\cap C'=\emp$; \label{DCdisjoint-1}
\item For any $C\in \AllCl$, $\{C'\in\AllCl\mid C\subset C'\}$ is a finite $\subseteq$-chain. \label{finite-chainCUp}

\hide{~\\
WRONG
\item If $C\neq C'$, then $D(C)$ and $D(C')$ and disjoint.\label{DCdisjoint-2Wrong}
}
\end{enumerate}
\end{claim}
\begin{proof}
To prove the first statement, take $c$ in $C\cap C'$. Let $n=\ind{V},~n'=\ind{V'}$.
Then $$V=[c]_n\subseteq [c]_{n'}=V'.$$

It follows that $\overline{V}\subseteq \ov{V'}$: indeed,
if $a\in \overline{V}$, then $\ind{a}>n$ by Claim \ref{claim:ind-a-greater} and the definition of $\overline{V}$,
and so $\ind{a}>n'$, which implies $a\in \overline{V'}$.
Hence, the relation $R\restr \overline{V}$ is included in the relation $R\restr \overline{V'}$.
Assume $c'\in C$. Then  we have $(c,c') \in (R\restr \overline{V})^*$, so
$(c,c') \in (R\restr \overline{V'})^*$; the latter means that $c'\in C'$, since
$C'$ is maximal in $G_{V'}$. Hence $C\subseteq C'$, which
completes the proof of the first statement.

The second statement is immediate from the first one:
if there is $c$ in $C\cap C'$, consider $V,V'\in \AllDef$ such that
$C\in \Cl{V}$ and $C'\in \Cl{V'}$, and compare $\ind{V}$ and $\ind{V'}$.

That $\{C'\in\AllCl\mid C\subset C'\}$ is a chain follows from \eqref{DCdisjoint-1}.
Let $$\Delta=\{V'\in\AllDef\mid \EE C'\in \Cl{V'}\, (C\subset C')\}.$$
Notice that every $V'$ from $\Delta$ includes $C$, so by \eqref{eq:AllPart}, $\Delta$ is a $\subseteq$-chain.
We have $C\in \Cl{V}$ for some $V\in \AllDef$.
By \eqref{AllC-ind-monot},  $\ind{V'}<\ind{V}$ for every $V'\in \Delta$, so the chain $\Delta$ is finite.
Observe that for every $V'\in\Delta$, there is a unique $C'\in \Cl{V'}$ such that
$C\subset C'$: indeed, distinct clusters are disjoint;
define $f(V')=C'$. So $f$ is a function from the finite set $\Delta$ to $\{C'\in\AllCl\mid C\subset C'\}$. By
the definition of $\clC$ and $\Delta$, $f$ is surjective, which completes the proof.
\end{proof}

Consider the poset $(\AllCl,\subseteq)$.
Let $\clC_{\min}$ be the set of all minimal elements of $(\AllCl,\subseteq)$, and
$$
\clC_0=\{C\in \AllCl\mid \EE C'\in\clC_{\min}  \, (C'\subseteq C) \}, \quad \clC_1=\AllCl{\setminus}\clC_0.
$$

\begin{claim}\label{claim:mainForS}
There exist pairwise disjoint subsets $T^+$, $T^-$, $S^+$, $S^-$ of $X$ such that:
\begin{enumerate}[\normalfont(a)]
\item
for every $C\in \clC_0$, the sets $C\cap T^+$
and $C\cap T^-$ are non-empty, and
$T^+\cup T^-\subseteq \bigcup \clC_{\min}$;
\item
for every $C\in \clC_1$, the sets $C\cap S^+$
and $C\cap S^-$ are non-empty, and
$S^+\cup S^-\subseteq \bigcup \clC_1$.
\hide{
\item for every $C\in\AllCl$,
$C\cap (S^+\cup T^+)\neq \emp$ and $C\cap (S^-\cup T^-) \neq \emp$.}
\end{enumerate}
\end{claim}
\begin{proof}
First, we show existence of $T^+$ and $T^-$.
For $C$ in $\AllCl_{\min}$, let $t^+(C)$ and $t^-(C)$ be two distinct elements of $C$
(recall that $C$ has at least two elements).
Put $T^+=\{t^+(C)\mid C\in\AllCl_{\min}\}$ and $T^-=\{t^-(C)\mid C\in\AllCl_{\min}\}$.
Every $C'\in\clC_0$ includes some $C\in \clC_{\min}$, and so we have
$C'\cap T^+\neq \emp$ and  $C'\cap T^-\neq \emp$.
Every two distinct elements of $\clC_{\min}$ are disjoint by Claim
\ref{claim:DCdisjoint}(\ref{DCdisjoint-1}),
and so $T^+$ and $T^-$ are disjoint as well:
indeed, $t^+(C)=t^-(C')$ would imply $C'\neq C$ and at the same time $C'\cap C\neq\emp$.

Now we define $S^+$ and $S^-$.
For $C$ in $\clC_1$, let
$d(C)$ be the cardinality of the set $\Delta(C)=\{C'\in\clC_1\mid C\subset C'\}$.
By Claim \ref{claim:DCdisjoint}(\ref{finite-chainCUp}), $\Delta(C)$ is finite. 
By recursion on $d(C)$ we define
elements $a^+(C)$, $a^-(C)$ of $C\in\clC_1$.
Notice that $C$ is infinite, since $C\notin \clC_0$.
If $d(C)=0$, let $a^+(C),a^-(C)$ be two distinct elements in $C$.
Otherwise consider the set
$$C\setminus (\{ a^+(C')\mid C'\in\Delta(C)\}\cup \{a^-(C')\mid C'\in\Delta(C)\}).$$
Clearly, this set is infinite;
let $a^+(C),a^-(C)$ be two distinct elements in this set.
By the construction, we have
\begin{equation}\label{eq:disj-a}
\AA C'\in\Delta(C)  \,(a^+(C)\neq a^-(C') \;\&\; a^-(C)\neq a^+(C') ).
\end{equation}
Put $$S^+=\{a^+(C)\mid  C\in \clC_1\},~ S^-=\{a^-(C)\mid  C\in \clC_1\}.$$

Since $a^+(C),a^-(C)\in C$ for every $C$ in $\clC_1$,
the sets $C\cap S^+$, $C\cap S^-$ are non-empty
and
$S^+$ and $S^-$ are subsets of $\bigcup \clC_1$.

Let us check that $S^+$ and $S^-$ are disjoint. For the sake of contradiction,
assume that they are not. Then there are $C,C'\in \clC_1$ such that
$a^+(C)=a^-(C')$. Hence $C\neq C'$. Since
$a^+(C)\in C$ and $a^-(C')\in C'$,
 if follows that $C\cap C'\neq \emp$.  By Claim \ref{claim:DCdisjoint}(\ref{DCdisjoint-1}),
we have that $C\subset C'$ or  $C'\subset C$. 
Both cases contradict \eqref{eq:disj-a}: indeed, in the first case
$C'\in\Delta(C)$,
and in the second case
$C\in\Delta(C')$. This
contradiction proves that $S^+$ and $S^-$ are disjoint.

Notice that
\begin{equation}\label{eq:disj-cupC}
(\bigcup \clC_{\min})\cap (\bigcup \clC_1)=\emp;
\end{equation}
indeed, if $c\in C\in\clC_{\min}$ and
$c\in C'\in\clC_1$, then
by Claim
\ref{claim:DCdisjoint}(\ref{DCdisjoint-1}),
$C\subseteq C'$ or $C'\subset C$;
but the first implies that $C'\in\clC_0$, which contradicts the definition of $\clC_1$,
and the latter contradicts the minimality of $C$ in $(\AllCl,\subseteq)$.

It follows from \eqref{eq:disj-cupC} that  $T^+$, $T^-$, $S^+$, $S^-$ are pairwise disjoint.
\end{proof}

Consider sets $T^+, T^-, S^+, S^-$ satisfying the above claim. Put
$$S=T^+\cup S^+.$$
By the above claim, $C\cap (S^+\cup T^+) \neq \emp$ and
 $C\cap (S^-\cup T^-) \neq \emp$
for every $C\in\AllCl$.
Since $X{\setminus}S\supseteq T^-\cup S^-$, we have:
\begin{equation}\label{eq:mainPropOfS}
 \AA C\in\AllCl\, (C\cap S\neq \emp\, \&\, C\cap (X{\setminus}S) \neq \emp).
\end{equation}

\bigskip

\subparagraph{{\bf Embedding algebras}}
Since the logic of $G=(X,\refl{R},R_1,\ldots R_{l})$ is locally tabular,
the logic of the frame $(X,\refl{R})$ is locally tabular as well.
By Proposition \ref{prop:2finite-to-subs-m-n}, there is $h<\omega$ such that for every subset $V$ of $X$,
the height of $(V,\refl{R}\restr V)$ is not greater than $h$.
In particular, we have:
\begin{claim}\label{claim:finite_hOfGV}
For all $V\in\AllDef$, the height of
$G_V$ is finite.
\end{claim}
Recall that for a set $V\subseteq X$, $R_V$ denotes the restriction of $R$ to $V$,
and $R_V^*$ denotes its transitive reflexive closure.

The following claim is the key step towards expressing $Q_n$ in terms of $Q,E,S$.

\begin{claim}\label{claim:main}
Let $V\in \AllDef$, $n=\ind{V}$.
For every $a\in V$, we have:
$a\in Q_n$ iff
\begin{eqnarray}
  &&\text{$R_V^*(a)\subseteq Q$, and} \label{eq:cQ2} \hide{\text{Simpler: a in Q. No, what if maximal in $G_V$
sees it?}
}\\
  &&(\text{$R_V^*(a) \subseteq E$ or $R_V^*(a) \subseteq X{\setminus}E$})\text{, and  } \label{eq:C3}\\
  &&(\text{$R_V^*(a) \subseteq S$ or $R_V^*(a) \subseteq X{\setminus}S$}).  \label{eq:cS}
\end{eqnarray}
\end{claim}
\begin{proof}
``Only if''. Assume that
$a\in Q_n$.  
We have $R_V(a)=\emp$ by the definition of a defect, so $R^*_V(a)$ is the singleton $\{a\}$, which implies \eqref{eq:cQ2}, \eqref{eq:C3}, and \eqref{eq:cS}.

``If''.
Assume \eqref{eq:cQ2}, \eqref{eq:C3}, and \eqref{eq:cS} hold. From \eqref{eq:cQ2} we have $a\in Q$.
By Claim \ref{claim:ind-a-greater}, $\ind{a}\geq n$.

For the sake of contradiction, assume that $\ind{a}>n$, and show that
\begin{equation}\label{eq:contrMainStep}
\text{there exist $c,c'\in R_{\overline{V}}^*(a)$ such that $c\in S$ and $c'\notin S$.}
\end{equation}

Observe that $R_{\overline{V}}\subseteq R_V$.
Since $\ind{a}>\ind{V}$, $a$ belongs to $G_V$.
Since the height of $G_V$ is finite (Claim \ref{claim:finite_hOfGV}),
$R_{\overline{V}}^*(a)$ includes a maximal in $G_V$ cluster $C$.
Consider $b\in C$.
By Claim \ref{claim:ranks}, $\rk{b}\geq \rk{V}$, and if $\rk{b}>\rk{V}$, then
$bRb'$ for some $b'\in V\cap Q$ such that $bRb'$ and $\rk{b'}=\rk{b}-1$.
The latter is impossible due to \eqref{eq:C3}, since both $b$ and $b'$ are in $R_V^*(a)$, but exactly one of them is in $E$.
Hence,
\begin{equation}\label{eq:samerankBandV}
\text{$\rk{b}=\rk{V}$}.
\end{equation}
Since $b\in \overline{V}$, $\ind{b}>n$,
and so there exists $b'\in V$ with $bRb'$  by Claim \ref{claim:ind}.
By Claim \ref{claim:rks-monot},
$\Sigma^\up([b]_{\ind{b}})\supseteq \Sigma^\up(V)$, and so by
\eqref{eq:samerankBandV} we obtain $\Sigma^\up([b]_{\ind{b}})=\Sigma^\up(V)$.
It follows that
$V\notin\Sigma^\up([b]_{\ind{b}})$, and so
$b'$ is not a defect of $V$.
Hence, $b'$ is not in $Q_n$; due to \eqref{eq:cQ2}, $b'\in Q$ and so $b'\in \overline{V}$. Since
$b\in Q$, $(b,b)\notin R$. It follows that $b'\neq b$.
Since
$C$ is maximal in $G_V$, $b'\in C$. Since  $C$ is not a singleton, $C\in\AllCl$.
Now \eqref{eq:contrMainStep} follows from \eqref{eq:mainPropOfS}.

Since $R_{\overline{V}}\subseteq R_V$, \eqref{eq:contrMainStep} contradicts
\eqref{eq:cS}, which completes the proof of the claim.
\end{proof}

Consider the valuation $\theta':\{p_0,\ldots,p_{k-1}, q,e,s\}\to\clP(X)$ such that $\theta'$ extends
$\theta$ and  $\theta'(q)=Q$,  $\theta'(e)=E$, $\theta'(s)=S$.
Let $\sim$ be the equivalence induced in $(G,\theta')$ by
all formulas in variables $\{p_0,\ldots,p_{k-1}, q,e,s\}$.

\newcommand\psiq[1]{\psi_{#1}}
\def\psiqn{\psiq{n}}
\begin{claim}\label{claim:express-q-n}
There exist formulas $\{\psiqn\mid n<\omega\}$ in variables $\{p_0,\ldots,p_{k-1}, q,e,s\}$
such that
\begin{equation}\label{eq:psi-q-n}
 (G,\theta'),a\mo \psiqn \text{ iff } a\in Q_n.
\end{equation}
\end{claim}
\begin{proof}
By induction on $n$.
Assume that  $\psiq{i}$ are defined for all $i<n$.

Let $V\in X/{\sim_n}$.
Since $X/{\sim_n}$ is finite (Claim \ref{claim:simi-are-finandtuned}), there is
a formula $\vf(V)$ in variables $p_0,\ldots,p_{k-1}$ and $q_{i}$ with $i<n$ such that
$$(G,\theta_n),a\mo \vf(V) \text{ iff } a\in V. $$
By induction hypothesis, there is a formula $\psi(V)$ in variables
$p_0,\ldots,p_{k-1},q,e,s$  such that
\begin{equation}\label{eq:expression_of_V}
(G,\theta'),a\mo \psi(V) \text{ iff } a\in V.
\end{equation}
By Proposition \ref{prop:relativised-m-trans},
there is $m<\omega$ such that
for every $a\in V$ we have: 
\begin{eqnarray*}
  (G,\theta'),a \mo \trSub{\Box_0^{\leq m}q}{\psi(V)} &\text{ iff }& R_V^*(a)\subseteq Q, \\
  (G,\theta'),a\mo \trSub{\Box_0^{\leq m}e}{\psi(V)}  \vee  \trSub{\Box_0^{\leq m}\neg e}{\psi(V)}   &\text{ iff } &
(R_V^*(a) \subseteq E \text{ or } R_V^*(a) \subseteq X{\setminus}E),\\
  (G,\theta'),a\mo \trSub{\Box_0^{\leq m} s}{\psi(V)} \vee \trSub{\Box_0^{\leq m}\neg s}{\psi(V)}  & \text{ iff } & (R_V^*(a) \subseteq S \text{ or }R_V^*(a) \subseteq X{\setminus}S).
\end{eqnarray*}
Let $\chi(V)$ be the conjunction of the three modal formulas mentioned in the above equivalences.
By Claim \ref{claim:main}, for every $V\in \AllDef$ with $\ind{V}=n$, for every $a\in V$,
\begin{equation}\label{eq:expression_of_Qn}
  (G,\theta'),a\mo \chi(V) \text { iff } a\in Q_n.
\end{equation}
Let $\clV$ be the set of defective sets in $X{/}{\sim_n}$. By Claim \ref{claim:simi-are-finandtuned}, $\clV$ is finite.
Put
$$\psiqn=\bigvee\{\psi(V)\con\chi(V)\mid V\in \clV\}.$$
We have
$Q_n=\bigcup\{\Def{V} \mid V\in \clV\}$, and
so \eqref{eq:psi-q-n} follows from \eqref{eq:expression_of_V} and \eqref{eq:expression_of_Qn}.
\end{proof}

From Claim \ref{claim:express-q-n}, it follows that all
$Q_n$ belong to the subalgebra $B$ of $\Alg{G}$ generated by
$P_0,\ldots P_{k-1}, Q,E,S$. The algebra $B$ is finite due to the local finiteness of $\Alg{G}$.
Since all $Q_n$  are disjoint (Claim \ref{claim:Qndisjoint}), it follows that
there is $N<\omega$ such that $Q_N=\emp$.
The partition $X{/}{\sim}_N$ is finite and tuned in
$G$ by Claim \ref{claim:simi-are-finandtuned}, and since there are no defective $N$-classes, it is tuned in
$F$ due to Claim \ref{claim:noDef-tuned}.

Finally, consider the initial algebra $A$, the subalgebra of $\Alg{F}$ generated by $P_0,\ldots, P_{k-1}$.
By Proposition \ref{prop:betweenSizeOfAlgAndPart}\eqref{prop:SizePartToOfAlg}, every element of $A$ is the union of some $\sim_N$-classes.
Every such union is in $B$, so
$
\dom{A}\subseteq \dom{B}.
$
\end{proof}
\begin{lemma}\label{lem:refl-reduction-key-poly}  If the logic of an $\Al$-frame $\refl{\frF}$ is locally tabular, then
for every $k<\omega$, every $k$-generated subalgebra of $\Alg{F}$
is contained in a $(k+3|\Al|)$-generated subalgebra of $\Alg{\refl{F}}$.
\end{lemma}
\begin{proof}
Consider an $(n+m)$-frame $F=(X,(R_i)_{i<n},(S_i)_{i<m})$ such that all $S_i$ are reflexive,
and $\Log{\refl{F}}$ is locally tabular.
By induction on $n$, we show that every $k$-generated subalgebra of $\Alg{F}$
is contained in a $(k+3n)$-generated subalgebra of $\Alg{\refl{F}}$.
Basic $n=0$ is trivial. Let  $n>0$. Consider the frame $G=(X,(R_i)_{i<n-1},\refl{R_n},(S_i)_{i<m})$.
Notice that $\refl{G}=\refl{F}$.
By induction hypothesis,
every $k$-generated subalgebra of $\Alg{G}$
is contained in a $(k+3(n-1))$-generated subalgebra of $\Alg{\refl{F}}$.
In particular, it follows that $\Log{G}$ is locally tabular.  By
Lemma \ref{lem:refl-reduction-key-uni}, every $k$-generated subalgebra of $\Alg{F}$
is contained
in a $(k+3)$-generated subalgebra of $\Alg{G}$, and consequently,  in a
$(k+3n)$-generated subalgebra of $\Alg{\refl{F}}$.  This completes the proof of the induction step.

To complete the proof of the lemma, put $m=0$.
 \end{proof}

\begin{theorem}\label{thm:ref-to-irref} Let $\clF$ be a class of $\Al$-frames.
Then $\Log{\clF}$ is locally tabular iff $\Log{\refl{\clF}}$ is locally tabular.
\end{theorem}
\begin{proof}
The ‘only if’ direction is clear due to the following well-known translation:
put
$\reflTr{\bot}=\bot$, $\reflTr{p}=p$ for variables, $\reflTr{\vf\imp \psi}=\reflTr{\vf}\imp \reflTr{\psi}$,
and $\reflTr{\Di \psi}= \Di\reflTr{\psi}\vee \reflTr{\psi}$ for $\Di\in\Al$.
It is straightforward that
$$\clF\mo\reflTr{\vf} \text{ iff } \refl{\clF}\mo\vf.$$
So the local tabularity of $\Log{\clF}$ implies the local tabularity of $\Log{\refl{\clF}}$.

The ‘if’ part follows from Lemma \ref{lem:refl-reduction-key-poly} and Theorem
\ref{thm:Maltsev}\eqref{thm:Maltsev-kLF}. Namely,
since $\Log{\refl{\clF}}$ is locally tabular, the class $\clC$ of algebras
$\{\Alg{F}\mid F\in\refl{\clF}\}$ is uniformly locally finite, that is
there exists a function $f:\omega\to \omega$ such that the cardinality of a subalgebra of any $A\in \clC$ generated by
$k<\omega$ elements does not exceed $f(k)$. It follows that
the class $\clD=\{\Alg{F}\mid F\in\clF\}$ is uniformly locally finite as well:
for every algebra $B$ in $\clD$, the size of its subalgebra generated by $k$ elements   is bounded by $f(k+3|A|)$, according to Lemma \ref{lem:refl-reduction-key-poly}.
\end{proof}

We do not know if an analogous result holds for the case of locally finite algebras.
Finite height was crucial in the proof of Lemma \ref{lem:refl-reduction-key-uni} (see
Claim \ref{claim:finite_hOfGV}). However,
local finiteness of an algebra does not imply this property, because, for example, the algebra
of the frame $(\omega,\leq)$ is locally finite (follows easily from Theorem \ref{thm:LFviaTuned}\eqref{thm:LFviaTuned-lfA}).

\section{Transfer results}\label{sec:transfer}
For disjoint alphabets of modalities $\AlA$ and $\AlB$, and logics $L_1$ in the language of $\Al$ and $L_2$ in the language of $\AlB$, the {\em fusion $L_1*L_2$ of $L_1$ and $L_2$} is the smallest logic in the language of $\AlA\cup\AlB$ that contains $L_1\cup L_2$. In the case of consistent $L_1$ and $L_2$,
the fusion is a conservative extension of $L_1$ and of $L_2$ \cite{Thomason1980-yj}, and hence
local tabularity of $L_1$ and $L_2$ is a necessary condition for local tabularity of the fusion.
But it is not sufficient.
While it is known that Kripke completeness, finite model property,
and decidability of logics are preserved under
this operation \cite{KrachtWolter1991,fine1996transfer, Wolter1996fusions}, this is not the case for local tabularity. For example, let $\Al=\{\Di_a\}$, $\AlB=\{\Di_b\}$; it is
straightforward that the fusion $\LS{5}*\LS{5}$ of two instances of a locally tabular logic $\LS{5}$ is not even 1-finite
(one can easily construct a non-pretransitive
frame of $\LS{5}*\LS{5}$).
Notice however that every proper extension of the logic
$\LS{5}*\LS{5}+\{\Di_a\Di_b p \leftrightarrow \Di_b \Di_a p\}$ is locally tabular \cite{NickS5}.
This gives examples of locally tabular {\em products of modal logics}
\cite{GabbShehtProdI}.
Another family of locally tabular modal products was identified
in \cite{Shehtman2012}, see also \cite{Shehtman2018}. For related systems, {\em intuitionistic modal logics} (considered in the sense of \cite[Section 10.3]{zakharyaschev_many-dimensional_2003}),
certain locally tabular families were identified in \cite{Guram98MIPCI}, \cite{Guram-Revaz}, and \cite{Bezhanishvili2001}.\footnote{See also a recent paper \cite{Bezhanishvili2023local}.}
While the product of two modal logics inherits
both pretransitivity and finite height, the local tabularity is not preserved in general.\footnote{
A characterization of locally tabular products of modal logics was
provided in a recent manuscript \cite{ShapSl2024}.}

\smallskip
Below we consider operations on frames and on logics  which preserve local tabularity.

\subsection{First semantic transfer result}

\begin{definition}\label{def:sum-poly-extra}
Consider an $\Al$-frame $\frI=(Y,(S_\Di)_{\Di\in \Al})$
and a family $(\frF_i)_{i\in Y}$ of $\Al$-frames
$\frF_i=(X_i,(R_{i,\Di})_{\Di\in \Al})$.
The {\em sum
 $\LSum{\frI}{\frF_i}$ of the family $(\frF_i)_{i\in Y}$ over
 $\frI$}
 is the $\Al$-frame $(\bigsqcup_{i \in Y} X_i, (R^\Sigma_\Di)_{\Di\in \Al})$,
 where
 $\bigsqcup_{i\in Y}{X_i}=\bigcup_{i\in Y}(\{i\}\times X_i)$
is the  disjoint union of sets $X_i$, and
$$(i,a)R^\Sigma_\Di (j,b) \quad \tiff \quad (i = j \,\&\,a R_{i,\Di} b) \text{ or } (i\neq j \,\&\, iS_\Di j).$$
For  classes $\clI$, $\clF$  of $\Al$-frames,
let $\LSum{\clI}{\clF}$  be the class of all sums
$\LSum{\frI}{\frF_i}$ such that $\frI \in \clI$ and  $\frF_i\in \clF$ for every $i$ in $\frI$.
\end{definition}

A typical example of sum is a preorder, which is, up to an isomorphism, the sum
of its clusters over its skeleton. 

According to the definition, the relations $R^\Sigma_\Di$ are independent of reflexivity of the relations $S_\Di$:
\begin{equation}\label{eq:sumIndependentOfRefl}
  \LSum{\frI}{\frF_i}=\LSum{\refl{\frI}}{\frF_i}=\LSum{\irrefl{\frI}}{\frF_i},
\end{equation}
where $\irrefl{\frI}=(Y,(\irrefl{S_\Di})_{\Di\in \Al})$, $\irrefl{S_\Di}=S_\Di{\setminus}\{(a,a)\mid a\in Y\}$.

\begin{theorem}\label{thm:Main1}
Let $\clF$ and $\clI$ be classes of $\Al$-frames.
If the modal logics $\Log{\clF}$ and
$\Log{\clI}$ are locally tabular, then the logic $\Log{\LSum{\clI}{\clF}}$  is locally tabular as well.
\end{theorem}
\begin{proof}
Let $\clI_1=\{\irrefl{\frI}\mid \frI\in \clI\}$.
By \eqref{eq:sumIndependentOfRefl}, $\LSum{\clI}{\clF}=\LSum{\clI_1}{\clF}$.
Since $\refl{\clI}=\refl{(\clI_1)}$, it follows from Theorem \ref{thm:ref-to-irref}
that $\Log{\clI_1}$ is locally tabular.

Let $\clF_1$ be the class of all frames of the logic $\Log{\clF}$.
By Theorem \ref{thm:LFviaTuned}\eqref{thm:LFviaTuned-lfLog}, $\clF_1$ and $\clI_1$ are uniformly tunable, that is,
there are functions $f,g:\omega\to \omega$ such that $\clF_1$ is $f$-tunable and $\clI_1$ is $g$-tunable.
Without loss of generality, we can assume that $f$ and $g$ are monotone.
We will show that the class $\LSum{\clI_1}{\clF}$ is uniformly tunable.

Consider a sum $\LSum{\frI}{\frF_i}$,
where all $\frF_i=(X_i,(R_{i,\Di})_{\Di\in \Al})$ are in $\clF$, $\frI=(Y,(S_\Di)_{\Di\in \Al})\in\clI_1$. Then $\bigsqcup_{i\in Y}{X_i}$ is
the domain of $\LSum{\frI}{\frF_i}$.
Let $\clV_0$ be a finite partition of $\bigsqcup_{i\in Y}{X_i}$.
Since all $\frF_i$ are in $\clF$, the disjoint sum $\bigsqcup_{i\in Y} \frF_i$ is in $\clF_1$. Hence, there exists a refinement $\clV$ of $\clV_0$
such that $\clV$ is tuned in $\bigsqcup_{i\in Y}{\frF_i}$
and
\begin{equation}\label{eq:sizeOfV}
|\clV|\leq f(|\clV_0|).
\end{equation}


Consider the following function $t: Y\to \clP(\clV)$:
for $i\in Y$, put $$t(i)=\{V\in\clV\mid \EE a\in X_i\, (i,a)\in V\}.$$
On $Y$, put $$i\sim  j \text{ iff } t(i)=t(j),$$
and let $\clU_0$ be the quotient $Y{/}{\sim}$.
By the construction,
\begin{equation}\label{eq:sizeOfU_0}
|\clU_0|\leq 2^{|\clV|}.
\end{equation}
Since $\clI_1$ is $g$-tunable, there exists
a finite refinement $\clU$ of $\clU_0$ such that  $\clU$  is tuned in $\frI$ and
\begin{equation}\label{eq:sizeOfU}
|\clU|\leq g(|\clU_0|).
\end{equation}
Let $\clS$ be the quotient of $\bigsqcup_{i\in I}{X_i}$ by the following equivalence:
$$(i,a)\sim_\clS (i',a') \text{ iff } (i,a)\sim_\clV (i',a') \text{ and } i\sim_\clU i'.$$
Clearly, $\clS$ refines $\clV_0$, and also we have \begin{equation}\label{eq:sizeS}
|\clS|\leq |\clV|\times |\clU|.
\end{equation}

Define the function $h:\omega\to\omega$ as follows:
$$h(n)=f(n)\cdot g(2^{f(n)}).$$
According to \eqref{eq:sizeS}, \eqref{eq:sizeOfV}, \eqref{eq:sizeOfU_0}, and \eqref{eq:sizeOfU}, we have
\begin{equation}\label{eq:sizeOfS}
|\clS|\leq h(|\clV_0|).
\end{equation}

We claim  that $\clS$ is tuned in $\LSum{\frI}{\frF_i}$.
Assume that $(i,a)\sim_\clS (i',a')$ and
$(i,a)R^\Sigma_\Di (j,b)$. We need to show that there exists $(j',b')$ such that
\begin{equation}\label{eq:tunedS}
(i',a')R^\Sigma_\Di (j',b') \text{ and }(j,b)\sim_\clS (j',b').
\end{equation}

Consider two cases.

Case 1. Assume that $i\neq j$. Then we have $i S_\Di j$.
Since $(i,a)\sim_\clS (i',a')$, we have $i\sim_\clU i'$. Since $\clU$ is tuned in $\frI$, we have $i'S_\Di j'$ and
$j\sim_\clU j'$  for some $j'$.
Since $\clU$ refines $\clU_0$,
we have $t(j)=t(j')$. Let $V$ be the element of $\clV$ that contains $(j,b)$. Then $V\in t(j')$, and thus
there is $b'$ in $X_{j'}$ such that $(j',b')\in V$.
It follows that $(j,b)\sim_\clS (j',b')$.
Since $i'S_\Di j'$ and  $S_\Di$ is irreflexive, we have $i'\neq j'$. Thus
$(i',a')R^\Sigma_\Di (j',b')$, and so \eqref{eq:tunedS} follows.


Case 2. Assume that $i=j$. Since $(i,a)\sim_\clS (i',a')$, we have $(i,a)\sim_\clV (i',a')$.
By the definition of $R^\Sigma_\Di$, we have $a R_{i,\Di} b $.
Since $\clV$ is tuned in $\bigsqcup_{i\in Y}{\frF_i}$,
we have $(i,b)\sim_\clV (i',b')$ and $a'R_{i',\Di} b'$
for some $b'$ in $X_{i'}$.
Also, we have $i\sim_\clU i'$. Put $j'=i'$.
By the definition of $R^\Sigma_\Di$, we get
$(i',a')R^\Sigma_\Di (j',b')$, which proves \eqref{eq:tunedS}.

It follows that $\clS$ is tuned. According
to \eqref{eq:sizeOfS},
$\LSum{\clI_1}{\clF}$ is uniformly tunable.
\end{proof}


The construction in the above proof also shows the following.
\begin{corollary}
Let $\frI=(Y,(S_\Di)_{\Di\in \Al})$ be an $\Al$-frame where all $S_\Di$ are irreflexive.
If $(\frF_i)_I$ is a family of $\Al$-frames, and two algebras $\Alg{\bigsqcup_{i\in Y} \frF_i}$ and $\Alg{\frI}$
are locally finite,
then $\Alg{\LSum{\frI}{\frF_i}}$ is locally finite and hence the logic of
$\LSum{\frI}{\frF_i}$ has the finite model property.
\end{corollary}
We do not know if this corollary remains true without the irreflexivity condition.

\subsection{Second semantic transfer result}
The sum operation given above
does not change the signature.
The following sum-like operation combines modal languages.
Let $\AlA$ and $\AlB$ be finite disjoint sets.
Elements of $\AlA$ and $\AlB$ will be addressed as {\em vertical} and {\em horizontal modalities}, respectively.

In this and the following subsection we always assume that
$L_1$ and $L_2$ are logics in the alphabets
of $\AlA$ and $\AlB$, respectively.

\begin{definition}\label{def:sumLex}
Let $\frI=(Y,(S_\Di)_{\Di\in \AlA})$ be an $\AlA$-frame, $(\frF_i)_{i\in Y}$ a family of $\AlB$-frames,
  $\frF_i=(X_i,(R_{i,\Di})_{\Di\in \AlB})$.
The {\em lexicographic sum} $\LSuml{\frI}{\frF_i}$ is the $(\AlA\cup\AlB)$-frame
$\left(\bigsqcup_{i \in Y} X_i, (S^\oplus_\Di)_{\Di\in\AlA}, (R_\Di)_{\Di\in\AlB}\right)$,  where
\begin{eqnarray*}
\text{for }\Di\in\AlA,&& (i,a) S^\oplus_\Di  (j,b) \text{ iff }   i S j;\\
\text{for }\Di\in\AlB,&& (i,a) R_\Di  (j,b)  \text{ iff }    i = j \;\&\;a R_{i,\Di}  b.
\end{eqnarray*}
For a class  $\clI$ of $\AlA$-frames and a class  $\clF$ of $\AlB$-frames, $\LSuml{\clI}{\clF}$ denotes
 the class of all sums
$\LSuml{\frI}{\frF_i}$, where $\frI \in \clI$ and all $\frF_i$ are in $\clF$.
For modal logics $L_1$ and $L_2$,
let $\LSuml{L_1}{L_2}$ be the logic of the class
$\LSuml{\clF_1}{\clF_2}$, where $\clF_i$ is the class of all frames of the logic $L_i$.
\end{definition}
In the case when all summands are equal, this operation is the {\em lexicographic
product of frames}; lexicographic products of modal logics were
introduced in \cite{Balb2009}.

According to the following observation, lexicographic sum can be considered as a particular case of sum in the sense of Definition \ref{def:sum-poly-extra}.
Let $(\emp)_\AlB$ denote the sequence of length $\AlB$ consisting of empty sets.
\begin{proposition}\label{prop:lexviasum}
Let $\frI=(Y,(S_\Di)_{\Di\in \AlA})$ be an $\AlA$-frame, $(\frF_i)_{i\in Y}$ a family of $\AlB$-frames, where
  $\frF_i=(X_i,(R_{i,\Di})_{\Di\in \AlB})$.
Then
$$\textstyle
\LSuml{\frI}{\frF_i}\; = \;\LSum{\frI'}{(X_i, (T_{i,\Di})_{\Di\in\AlA},  (R_{i,\Di})_{\Di\in \AlB})},
$$
where
$\frI'$ is the $(\AlA\cup\AlB)$-frame $(Y,(S_\Di)_{\Di\in\AlA},(\emp)_\AlB)$, and for $\Di\in\AlA$, $i\in Y$, we put
$T_{i,\Di}=X_i\times X_i$ whenever
$(i,i)\in S_\Di$, and $T_{i,\Di}=\emp$ otherwise.
\end{proposition}
\begin{proof}
Straightforward  by Definitions \ref{def:sum-poly-extra} and \ref{def:sumLex}.
\end{proof}

\begin{proposition}\label{prop:idlemodsTunable}
Let $\frF=(X,(R_\Di)_{\Di\in \AlB})$ and
$\frF'=(X,(T_\Di)_{\Di\in \AlA},(R_\Di)_{\Di\in \AlB})$ be frames, and
$T_\Di=\emp$ or $T_\Di=X\times X$ for every $\Di\in\AlA$.
If a partition $\clV$ of $X$ is tuned in $\frF$, then
$\clV$ is  tuned in $\frF'$.
\end{proposition}
\begin{proof}
Every partition is tuned with respect to the empty relation, and  with respect to the relation $X\times X$.
\end{proof}

\begin{theorem}\label{thm:localfin-lex-semantically}
If the logics $L_1$ and $L_2$ are locally tabular, then
the logic $\LSuml{L_1}{L_2}$  is locally tabular as well.
\end{theorem}
\begin{proof}
Let $\clI$ and $\clF$ be the classes of frames of the logics $L_1$ and $L_2$, respectively, $L$ denote the logic $\LSuml{L_1}{L_2}$.
Let
\begin{eqnarray*}
  \clI' &=& \{(Y,(S_\Di)_{\Di\in\AlA},(\emp)_\AlB)\mid  (Y,(S_\Di)_{\Di\in\AlA})\in \clI\},
\end{eqnarray*}
and let $\clF'$ be the class of frames $(X, (T_{\Di})_{\Di\in\AlA},  (R_{\Di})_{\Di\in \AlB})$
such that
$(X, (R_{\Di})_{\Di\in \AlB})$ is in $\clF$, and for all $\Di\in\AlA$,  $T_{\Di}=\emp$ or $T_{\Di}= X\times X$.
By Proposition \ref{prop:lexviasum},
$\LSuml{\clI}{\clF}\;\subseteq \;\LSum{\clI'}{\clF'}.$
By Theorem \ref{thm:LFviaTuned}\eqref{thm:LFviaTuned-lfLog}, classes $\clI$ and $\clF$ are uniformly tunable. Hence, by Proposition \ref{prop:idlemodsTunable},
classes $\clI'$ and $\clF'$ are uniformly tunable as well.  By Theorem \ref{thm:Main1},
the logic $L'$ of the class $\LSum{\clI'}{\clF'}$ is locally tabular. Since $L'\subseteq L$,
the statement follows.
\end{proof}

\subsection{Axiomatic transfer result}
Let
$\Phi(\AlA,\AlB)$ be the set of all formulas
\begin{equation}\label{eq:alpha-beta-gamma}
\Dih \Div p\to \Div p, \; \Div\Dih p\to \Div p, \;\Div p\to \Boxh \Div p
\end{equation}
with $\Div$ in $\AlA$ and $\Dih$ in $\AlB$.
The corresponding frame conditions are the following:
$(X,(R_\Di)_{\Di\in\AlA\cup\AlB})\mo \Phi(\AlA,\AlB)$
iff for all $\Div$ in $\AlA$ and $\Dih$ in $\AlB$,
\begin{eqnarray}
R_\Dih\circ R_\Div \subseteq R_\Div,\label{eq:alpha}\\
R_\Div\circ R_\Dih \subseteq R_\Div,\label{eq:beta}\\
{(R_\Dih)}^{-1}\circ R_\Div \subseteq R_\Div.\label{eq:gamma}
\end{eqnarray}

The formulas $\Phi(\AlA,\AlB)$ were considered in \cite{Balb2009} in connection with axiomatization problems
of lexicographic products,
  and also in \cite{Bekl-Jap}
in the context of polymodal provability logic.
\begin{proposition}[Follows from \cite{Balb2009}]
For a class $\clI$ of $\AlA$-frames and a class $\clF$ of $\AlB$ frames,
$$\LSuml{\clI}{\clF}\mo \Phi(\AlA,\AlB).$$
\end{proposition}
\begin{proof}
By Definition \ref{def:sumLex}.
\end{proof}
Let $L_1\oplus L_2$ be the extension of the fusion of $L_1$ and $L_2$ with the axioms $\Phi(\AlA,\AlB)$:
\begin{equation}\label{eq:lexAx-def}
  L_1\oplus L_2 = L_1 * L_2+\Phi(\AlA,\AlB).
\end{equation}
In many cases, we have
\begin{equation}\label{eq:lexAx-cor}
\LSuml{L_1}{L_2}=L_1\oplus L_2.
\end{equation}
For example,
(\ref{eq:lexAx-cor}) holds for the logic $\LSuml{\GL}{\GL}$ (where $\GL$ is the   G\"odel-L\"ob logic) \cite{Bekl-Jap} and, as follows from  \cite{Balb2009},
for
$\LSuml{\LS{4}}{\LS{4}}$.

Let $\frF=(X,(R_\Di)_\Al)$ and $\frG=(Y,(S_\Di)_\Al)$ be  frames.
Recall that $f:X\to Y$ is called a {\em p-morphism} {\em from $\frF$ to $\frG$}, if the following holds
for all $\Di\in\Al$:
\begin{eqnarray}
  &&\text{for all $a,b\in X$, if  $aR_\Di b$, then $f(a)S_\Di f(b)$, and }\label{eq:p-mor-mon}\\
  &&\text{for all $a\in X$, $u\in Y$, if  $f(a)S_\Di u$, then
$aR_\Di a'$ and $f(a')=u$ for some $a'\in X$.}   \label{eq:p-mor-lift}
\end{eqnarray}
The notation $f:\frF\toto \frG$ means that $\frG$ is the image of $\frF$ under a p-morphism $f$;
in this case $\Log{\frF}\subseteq \Log{\frG}$ (see, e.g., \cite[Theorem 3.14 (iii)]{BDV}).

For $a$ in $\frF$, we put $\cone{\frF}{a} = \frF\restr R^*_\frF(a)$; we say that
$\cone{\frF}{a}$ is {\em rooted} and $a$ is a {\em root} of $\frF$.

\begin{lemma}\label{lem:sum-is-contained-inPsi}
If $L_1\oplus L_2$ is Kripke complete, then $\LSuml{L_1}{L_2}\subseteq L_1\oplus L_2$.
\end{lemma}
\begin{proof}

Let $\frF=(X,(R_\Di)_{\Di\in\AlA\cup\AlB})$ be a frame with a root $r$,
$\frF\mo L_1\oplus L_2$.

Put $\frI=\cone{(X,(R_\Di)_{\Di\in\AlA})}{r}$.
For $i$ in $\frI$, put
$\frF_i=\cone{(X,(R_\Di)_{\Di\in\AlB})}{i}$.
Finally, put $$\frF^\oplus=\LSuml{\frI}{\frF_i}.$$
Observe that $I\mo L_1$ and $F_i\mo L_2$ for all $i$ in $I$.
Hence,  $\frF^\oplus$  validates $\LSuml{L_1}{L_2}$.

We will show that $\frF$ is a p-morphic image of $\frF^\oplus$.
Consider the unions of vertical and of horizontal relations in $\frF$:
\begin{eqnarray*}
&&\textstyle V=\bigcup_{\Di\in\AlA} R_\Di, \quad H =\bigcup_{\Di\in\AlB} R_\Di.
\end{eqnarray*}
Notice that we have the following inclusions between relations in $\frF$ for every vertical modality $\Di\in\AlA$:
\begin{eqnarray}
\label{eq:R-star-S-subS}
H^*\circ R_\Di &\subseteq &R_\Di,\\
\label{eq:S-R-star-subS}
R_\Di\circ H^* &\subseteq &R_\Di,\\
\label{eq:Rconv-star-S-subS}
{H^*}^{-1}\circ R_\Di &\subseteq &R_\Di.
\end{eqnarray}
Indeed, \eqref{eq:R-star-S-subS} easily follows from
\eqref{eq:alpha}, \eqref{eq:S-R-star-subS}  from \eqref{eq:beta}, and
\eqref{eq:Rconv-star-S-subS} from \eqref{eq:gamma}.

Let us check that
\begin{equation}\label{eq:cup-star-circ-star}
(V\cup H)^*=V^*\circ H^*.
\end{equation}
That  $(H\cup V)^*\supseteq H^*\circ V^*$ is immediate. To prove
the converse inclusion, by induction on $n$ we show that
$(H\cup V)^{n}\subseteq H^*\circ V^*$ for every $n\geq 0$. The case $n=0$ is trivial.
By the induction hypotheses,
$$
(V\cup H)^{n}\circ (V\cup H)\subseteq
(V^*\circ H^*) \circ (V\cup H)=
((V^*\circ H^*) \circ V)\cup ((V^*\circ H^*) \circ H).$$
By \eqref{eq:R-star-S-subS}, $H^* \circ V\subseteq V$.
So we have $(V^*\circ H^*) \circ V=V^*\circ (H^* \circ V)\subseteq V^*$, and hence
$$(V\cup H)^{n+1}\subseteq V^*\cup (V^*\circ (H^* \circ H))  =
V^*\circ H^*,$$
which proves \eqref{eq:cup-star-circ-star}.

For $(i,a)$ in $\frF^\oplus$, put $$f(i,a)= a.$$

Let us show that $f$ is surjective. Suppose $a\in X$.
Since $r$ is a root of $\frF$,
$r (V\cup H )^* a$.
By (\ref{eq:cup-star-circ-star}), we have $rV^*iH^*a$ for some $i$.
Then $i$ is in $\frI$, $a$ is in $\frF_i$, and so $(i,a)$ is in
$\frF^\oplus$.

Let us show that $f$ is a p-morphism.
Fix the following notation:
$$
 \frI=(Y,(S_\Di)_{\Di\in\AlA}),\quad
 \frF^\oplus=(X^\oplus, (S^\oplus_\Di)_{\Di\in\AlA},(R^\oplus_\Di)_{\Di\in\AlB}).
$$
By the construction, the frame $(X^\oplus, (R^\oplus_\Di)_{\Di\in\AlB})$ is
the disjoint sum of a family of generated
subframes of $(X,(R_\Di)_{\Di\in\AlB})$. It follows that
\begin{equation}\label{eq:p-morph-hor}
f:(X^\oplus, (R^\oplus_\Di)_{\Di\in\AlB}) \toto (X,(R_\Di)_{\Di\in\AlB}).
\end{equation}
Let us check that for every $\Di\in\AlA$,
\begin{equation}\label{eq:p-morph-vert}
f:(X^\oplus, S^\oplus_\Di)\toto (X,R_\Di).
\end{equation}
To check \eqref{eq:p-mor-mon}, assume that $(i,a)S^\oplus_\Di(j,b)$.
Then $iR_\Di j$, $iH^*a$, and $jH^*b$. Thus $a ({H^*}^{-1} \circ R_\Di \circ H^* )b$.
We have ${H^*}^{-1}\circ R_\Di\subseteq R_\Di$  by \eqref{eq:Rconv-star-S-subS},
and $R_\Di\circ H^*\subseteq R_\Di$ by \eqref{eq:S-R-star-subS},
so $aR_\Di b$.
To check \eqref{eq:p-mor-lift}, consider $(i,a)\in X^\oplus$ and suppose $a R_\Di b$. Then we have $i H^* a R_\Di b$. By (\ref{eq:R-star-S-subS}), $i R_\Di b$.
It follows that  $b$ is in $\frI$. Clearly, $b$ is in $\frF_b$, so $(b,b)$ is in
$\frF^\oplus$. By the definition of lexicographic sum,  we get $(i,a)S^\oplus_\Di (b,b)$.
Hence we have \eqref{eq:p-morph-vert}.
From \eqref{eq:p-morph-hor} and \eqref{eq:p-morph-vert}, we obtain that $f:\frF^\oplus\toto \frF$.

So every rooted frame $\frF$ of $L_1\oplus L_2$ is a p-morphic image of a frame of $\LSuml{L_1}{L_2}$,
which proves the lemma.
\end{proof}

\begin{theorem}\label{thm:trasfer-via-axioms}
Let $L_1$ and $L_2$ be locally tabular logics.  If $L_1\oplus L_2$  is  Kripke complete, then it is locally tabular.
\end{theorem}
\begin{proof}
By Lemma \ref{lem:sum-is-contained-inPsi}, $\LSuml{L_1}{L_2}\subseteq L_1\oplus L_2$.
By Theorem \ref{thm:localfin-lex-semantically},  the logic $\LSuml{L_1}{L_2}$ is locally tabular, and so $L_1\oplus L_2$  is.
\end{proof}

Formulas  \eqref{eq:alpha-beta-gamma} are Sahlqvist formulas (see, e.g., \cite[Section 10.3]{CZ}).
It follows that if the logics $L_1$ and $L_2$ are canonical (in the sense that they are valid in their
$\omega$-canonical frames), then
$L_1\oplus L_2$ is canonical, and so is Kripke complete.
\begin{corollary}
If $L_1$ and $L_2$ are canonical locally tabular logics, then $L_1\oplus L_2$ is locally tabular.
\end{corollary}
Notice that there are non-canonical locally tabular logics \cite[Section 6]{Goldblatt1995}.

\section{Acknowledgements}
I am grateful to the reviewer for multiple useful suggestions on an earlier version of the manuscript.

The work on this paper was partially supported by NSF Grant DMS - 2231414.

\bibliographystyle{amsalpha}
\bibliography{sumsLF}
\end{document}